\documentclass[11pt,reqno]{amsart}
\usepackage{a4wide}

\usepackage{mathrsfs}
\usepackage{amsfonts}
\usepackage{amsmath}
\usepackage{stmaryrd}
\usepackage{amssymb}
\usepackage{amsthm}
\usepackage{esint}

\usepackage{url}

\usepackage{amscd}
\usepackage{indentfirst}
\usepackage{enumerate}
\usepackage{graphicx}
\usepackage{appendix}

\usepackage[numbers,sort&compress]{natbib}

\usepackage{color}
\usepackage[colorlinks, linkcolor=blue, citecolor=blue]{hyperref}

\emergencystretch=\maxdimen
\newcommand{\pa}{\partial}

\renewcommand{\div}{{\rm div}}
\newcommand{\curl}{{\rm curl}}
\newcommand{\norm}[1]{\|#1\|}

\newtheorem{thm}{Theorem}[section]
\newtheorem{lem}[thm]{Lemma}
\newtheorem{pro}[thm]{Proposition}
\newtheorem{cor}[thm]{Corollary}
\newtheorem{rem}[thm]{Remark}

\numberwithin{equation}{section}

\begin{document}

\title[Nishida-Smoller type large solutions and exponential growth for CNS equations]
{Nishida-Smoller type large solutions and exponential growth for the compressible Navier-Stokes equations with slip boundary conditions in 3D bounded domain}

\author[S. Xu]{Saiguo Xu}
\address[Saiguo Xu]{School of Mathematics and Statistics, Guangxi Normal University, Guilin, Guangxi 541004, P.R. China.}
\email{xsgsxx@126.com}
\author[Y.H. Zhang]{Yinghui Zhang*}
\address[Yinghui Zhang] {School of Mathematics and Statistics, Guangxi Normal University, Guilin, Guangxi 541004, P.R.
China} \email{yinghuizhang@mailbox.gxnu.edu.cn}

\thanks{* Corresponding author.}

\thanks{This work was supported by National Natural Science
Foundation of China $\#$12271114, Guangxi Natural Science Foundation $\#$2024GXNSFDA010071, $\#$2019JJG110003, $\#$2019AC20214, Science and Technology Project of Guangxi $\#$GuikeAD21220114, the Innovation Project of Guangxi Graduate Education $\#$JGY2023061, Center for Applied Mathematics of Guangxi (Guangxi Normal University) and the Key Laboratory of Mathematical Model and Application (Guangxi Normal University), Education Department of Guangxi Zhuang Autonomous Region.}

\date{\today}

\begin{abstract}
This paper concerns the isentropic compressible Navier-Stokes equations in a three-dimensional (3D) bounded domain with slip boundary conditions and vacuum. It is shown that
the classical solutions to the initial-boundary-value problem of this system with large initial energy and vacuum  exist
globally in time and have an exponential decay rate which is decreasing with respect to the adiabatic exponent $\gamma>1$ provided that the fluid is nearly isothermal (namely, the adiabatic exponent is close enough to 1). This constitutes an extension of the celebrated result for the one-dimensional Cauchy problem of the isentropic Euler equations that has been established in 1973 by Nishida and Smoller (Comm. Pure Appl. Math. 26 (1973), 183-200). In addition, it is also shown that
the gradient of the density will grow unboundedly with an exponential rate when the initial vacuum appears (even at a point). 
 In contrast to previous related works,
where either small initial energy  are required or boundary effects are
ignored, this establishes the ﬁrst result on the global existence and exponential growth
 of large-energy solutions with vacuum to the 3D isentropic compressible
Navier–Stokes equations with slip boundary conditions.
\end{abstract}

\maketitle
{\small
\keywords {\noindent {\bf Keywords:} {Compressible Navier-Stokes equations; global existence; large initial energy; slip boundary conditions; vacuum.}
\smallskip
\newline
\subjclass{\noindent {\bf 2020 Mathematics Subject Classification:} 76N06; 76N10; 35M13; 35K65}
}

\section{Introduction}
The motion of a general viscous compressible, isentropic fluid in a domain $\Omega\subset\mathbb{R}^3$ is governed by the following compressible Navier-Stokes equations
\begin{equation}\label{Large-CNS-eq}
  \begin{cases}
    \pa_t\rho+\div(\rho u)=0, \\
    \pa_t(\rho u)+\div(\rho u\otimes u)-\mu\Delta u-(\mu+\lambda)\nabla\div u+\nabla P=0,
  \end{cases}
\end{equation}
where $\rho$, $u$, $P=a\rho^{\gamma}$ ($a>0$) denote the density, velocity and pressure respectively, $\gamma>1$ is the adiabatic exponent, $\mu$ and $\lambda$ represent the shear viscosity and bulk viscosity coefficients satisfying the physical restrictions:
\begin{equation}\label{viscosity-condition}
  \mu>0,\quad \lambda+\frac{2}{3}\mu\geq0.
\end{equation}
The equations \eqref{Large-CNS-eq} will be equipped with initial data
\begin{equation}\label{initial-data}
  (\rho, u)(x,0)=(\rho_0, u_0)(x), \quad x\in\Omega.
\end{equation}
In this paper, we assume that $\Omega\subset\mathbb{R}^3$ is a simply connected bounded domain with smooth boundary $\pa\Omega$. The system \eqref{Large-CNS-eq} is studied under the Navier-slip boundary conditions:
 \begin{equation}\label{Navier-slip-condition}
    u\cdot n=0,\, \curl u\times n=-An \textrm{ on }\pa\Omega,
 \end{equation}
 where $n$ is the unit outer normal to $\pa\Omega$, and $A=A(x)$ is a $3\times3$ symmetric matrix defined on $\pa\Omega$. There exist some different forms of slip boundary conditions related to \eqref{Navier-slip-condition}, where the detailed discussions can be found in \cite{Cai-Li2023}. 

There are large amounts of literature concerning the well-posedness theory for the compressible Navier-Stokes equations \eqref{Large-CNS-eq}. The one-dimensional problem has been investigated extensively by many people, see for instance \cite{Hoff1987,Kazhikhov1977,Serre1986-1,Serre1986-2,Kazhikhov1982} and references therein. The local well-posedness of multidimensional problem was studied by Nash \cite{Nash1962} and Serrin \cite{Serrin1959} in the absence of vacuum, while for the initial density with vacuum, the existence and uniqueness of local strong solution was proved in \cite{Cho2004,Cho2006-1,Cho2006-2,Choe2003,Salvi1993}. Matsumura and Nishida \cite{Matsumura1980} first proved the global existence of strong solutions with initial data close to the equilibrium, and later was extended to the discontinuous initial data by Hoff \cite{Hoff1991,Hoff1995}. For the existence of solutions for arbitrary large initial data in 3D, the major breakthrough was made by Lions \cite{Lions1998}, in which he showed global existence of weak solutions for the whole space, periodic domains, or bounded domains with Dirichlet boundary conditions with $\gamma>\frac{9}{5}$, and later was extended to the case  $\gamma>\frac{3}{2}$ by Feireisl \cite{Feireisl2001,Feireisl2004,Feireisl2004B}. When the initial data are
assumed to have some spherically symmetric or axisymmetric properties, Jiang and Zhang \cite{Jiang2001,Jiang2003}
proved the existence of global weak solutions for any $\gamma>1$. Shortly thereafter, Hoff \cite{Hoff2005} gave a new type of global weak solutions with small energy, which have extra regularity compared with the ones constructed by Lions-Feireisl in \cite{Lions1998,Feireisl2001} under an additional condition on viscosity coefficients and the far-field density $\tilde{\rho}>0$. 

Up to now, the uniqueness and regularity for weak solutions in \cite{Lions1998,Feireisl2001} (with arbitrarily large
initial data) still remains open. Recently, many important progress on global existence and uniqueness of classical solutions with large oscillations and vacuum to viscous compressible fluids in a barotropic regime have been made. Huang, Li and Xin \cite{Huang-Li2012} first established the global existence of classical solutions to 3D Cauchy problem of the isentropic compressible Navier–Stokes equations with small initial total energy but possibly large oscillations. Later, Li and Xin \cite{Li-Xin2019} studied the 2D Cauchy problem and the large time asymptotic behavior of solutions with small initial total energy. 
Very recently, Hong-Hou-Peng-Zhu \cite{Zhu2024} provided a positive result under the condition that the adiabatic exponent $\gamma$ is close to 1, which says that classical solutions to Cauchy problem of the isentropic compressible Navier-Stokes equations exist globally with allowing the large initial energy and the presence of vacuum. This type solution can be viewed as the Nishida-Smoller type large solution which is originally studied for the conservation laws with BV initial data in \cite{Nishida-Smoller1973}, where Nishida and Smoller showed the global existence of solutions to the Cauchy problem of 1D isentropic Euler equations under the condition that $(\gamma-1).\text{total var.}\{u_0,\rho_0\}$ is sufficiently small. In particular, this result implies that the initial energy could be large as $\gamma$ is sufficiently close to 1. For some generalizations of the Nishida-Smoller type results on inviscid or viscous flow, one can see for instance \cite{Kawashima-Nishida1981,Liu-Yang-Zhao2014,Tan-Yang-Zhao2013,Liu1977,Temple1981,Zhu2017}.
For compressible Navier-Stokes equations \eqref{Large-CNS-eq} in half space $\mathbb{R}^3_+$ with slip boundary conditions, Hoff \cite{Hoff2005} established the global existence of weak solutions under the assumption that the initial energy is suitably small. 
For compressible Navier-Stokes equations \eqref{Large-CNS-eq} in a non axis-symmetric domain, Novotn\'{y} and Str\v{a}skraba \cite{Novotny} proved global existence of weak solutions.
For compressible Navier-Stokes equations \eqref{Large-CNS-eq} in a general bounded smooth domain,
the global existence of strong (or classical) solutions has been investigated for the 3D case with slip boundary condition in \cite{Cai-Li2023}, and the 2D case with similar boundary condition in \cite{Fan-Li-Li2022}, both of which are equipped with small initial total energy and vacuum; even for the 3D bounded domain with non-slip boundary condition, Fan and Li \cite{Fan-Li-arX2021} proved the global existence of classical solutions to the barotropic compressible Navier-Stokes system with small initial energy. 
Also, one can refer to \cite{Cai-Li-Lv2021} for the exterior domain case. 

In conclusions, all the works \cite{Huang-Li2012, Li-Xin2019, Hoff2005, Cai-Li2023, Cai-Li-Lv2021, Nishida-Smoller1973, Zhu2024} depend essentially on small initial energy
or the advantage of the whole space. Therefore, a natural and important
problem is to study what will happen if both large initial energy and 
boundary effects are involved.
That is to say, to investigate the global well-posedness and
large time behavior of Nishida-Smoller type large solutions to compressible Navier-Stokes equations \eqref{Large-CNS-eq} with slip boundary conditions \eqref{Navier-slip-condition} and vacuum.
However, to the best of our knowledge, up to now, this problem still remains open.
The main purpose of this work is to resolve this problem.
More precisely, we prove the global existence and uniqueness of
classical solutions to the initial-boundary-value problem \eqref{Large-CNS-eq}-\eqref{Navier-slip-condition} with large initial energy and vacuum provided that the adiabatic exponent is sufficiently close to 1.
Moreover, we also prove that the classical solutions have an exponential decay rate which is decreasing with respect to the adiabatic exponent.
Finally, we prove that  the gradient of the density will grow unboundedly with an exponential rate if the vacuum appears (even at a point) initially. 
 This generalizes the previous related works in \cite{Huang-Li2012, Li-Xin2019, Hoff2005, Cai-Li2023, Cai-Li-Lv2021, Nishida-Smoller1973, Zhu2024},
where either small initial energy  are required or boundary effects are
absent.

\bigskip

Before stating our result, let us introduce the following notations and conventions used throughout this paper. We set
\begin{equation*}
  \int f=\int_{\Omega}fdx,\quad\int_0^Tg=\int_0^Tgdt
\end{equation*}
and
\begin{equation*}
  \bar{f}:=\frac{1}{|\Omega|}\int_{\Omega}fdx
\end{equation*}
which is the average of $f$ on $\Omega$.

For $1\leq r\leq\infty$, and integer $k\geq1$, we denote the standard Sobolev spaces as follows:
\begin{equation}\label{Notation-Sobolev}
  \begin{cases}
     L^r=L^r(\Omega),\,D^{k,r}=\{u\in L_{loc}^1(\Omega): \norm{\nabla^ku}_{L^r}<\infty\},\\
     W^{k,r}=L^r\cap D^{k,r},\,H^k=W^{k,2},\, D^k=D^{k,2},\\
     D_0^1=\{u\in L^6: \norm{\nabla u}_{L^2}<\infty,\textrm{ and \eqref{Navier-slip-condition} holds}\},\\
    H_0^1=L^2\cap D_0^1,\,\norm{u}_{D^{k,r}}=\norm{\nabla^ku}_{L^r}.
  \end{cases}
\end{equation}
For some $s\in(0,1)$, the fractional Sobolev space $H^s(\Omega)$ is defined by
\begin{equation*}
  H^s(\Omega):=\left\{u\in L^2(\Omega): \int_{\Omega\times\Omega}\frac{|u(x)-u(y)|^2}{|x-y|^{n+2s}}dxdy<\infty\right\}
\end{equation*}
with the norm:
\begin{equation*}
  \norm{u}_{H^s(\Omega)}:=\norm{u}_{L^2(\Omega)}+\left(\int_{\Omega\times\Omega}\frac{|u(x)-u(y)|^2}{|x-y|^{n+2s}}dxdy\right)^{\frac{1}{2}}.
\end{equation*}
The initial total energy of \eqref{Large-CNS-eq} is defined as
\begin{equation}\label{defi-total-energy}
  E_0:=\int(\frac{1}{2}\rho_0|u_0|^2+\frac{1}{\gamma-1}P(\rho_0)),
\end{equation}
and the modified initial energy involving $\gamma-1$ is denoted as
\begin{equation}\label{defi-modified-energy}
  \mathcal{E}_0:=\int\frac{1}{2}\rho_0|u_0|^2+(\gamma-1)E_0.
\end{equation}
In the following, we denote by $C>0$ a generic constant depending on $\mu, \lambda, a, \tilde{\rho}, \Omega, M$ and the matrix $A$, but independent of $\gamma-1, E_0, \mathcal{E}_0$ and $t$. And we write $C(\alpha)$ to emphasize the dependence of $C$ on the parameter $\alpha$. 

Now, we are in a position to state our main results.
\begin{thm}\label{thm-global-CNS}
 Let $\Omega$ be a simply connected bounded domain in $\mathbb{R}^3$ and its smooth boundary $\pa\Omega$ has a finite number of 2-dimensional connected components. For given positive constants $M$ and $\tilde{\rho}$, suppose that the $3\times3$ symmetric matrix $A$ in \eqref{Navier-slip-condition} is smooth and positive semi-definite, and the initial data $(\rho_0,u_0)$ satisfy for some $q\in(3,6)$,
 \begin{equation}\label{initial-data1}
   (\rho_0,u_0)\in W^{2,q},\quad u_0\in\{f\in H^2: f\cdot n=0,\,\curl f\times n=-Af\textrm{ on }\pa\Omega\},
 \end{equation}
 \begin{equation}\label{initial-data2}
   0\leq\rho_0\leq\tilde{\rho},\quad\norm{\nabla u_0}_{L^2}\leq M,
 \end{equation}
 and the compatibility condition
 \begin{equation}\label{compatibility-condition}
   -\mu\Delta u_0-(\mu+\lambda)\nabla\div u_0+\nabla P(\rho_0)=\rho_0^{\frac{1}{2}}g,
 \end{equation}
 for some $g\in L^2$. Then the initial-boundary value problem \eqref{Large-CNS-eq}-\eqref{Navier-slip-condition} admits a unique classical solution $(\rho,u)$ in $\Omega\times(0,\infty)$ satisfying that
 \begin{equation}\label{density-bound}
   0\leq\rho(x,t)\leq2\tilde{\rho},\,(x,t)\in\Omega\times(0,\infty),
 \end{equation}
 and for any $0<\tau<T<\infty$,
 \begin{equation}\label{classical-sol}
   \begin{cases}
     (\rho,P)\in C([0,T];W^{2,q}),\\
     \nabla u\in C([0,T];H^1)\cap L^{\infty}(\tau,T;W^{2,q}),\\
     u_t\in L^{\infty}(\tau,T;H^2)\cap H^1(\tau,T;H^1),\\
     \sqrt{\rho}u_t\in L^{\infty}(0,\infty;L^2),
   \end{cases}
 \end{equation}
 provided
 \begin{equation}\label{small-condition}
   \mathcal{E}_0\leq\epsilon, \quad \norm{A}_{W^{1,6}}\leq\hat{\epsilon}.
 \end{equation} 
 Here $\epsilon>0$ is a small constant depending only on $\mu, \lambda, \gamma, a, \tilde{\rho}, \Omega, M, E_0$, and the matrix $A$, but independent of $\gamma-1$ and $t$, $\hat{\epsilon}>0$ is a small constant depending only on $\mu, \lambda$ and $\Omega$. According to \eqref{small-assumption1}, \eqref{small-assumption2}, \eqref{small-assumption3}, \eqref{small-assumption4} and \eqref{small-A-assumption1}, $\epsilon$ and $\hat{\epsilon}$ can be respectively precisely characterized in the following form:
 \begin{equation*}
 \begin{aligned}
   \epsilon&=\min\left\{1,(4C(\tilde{\rho}))^{-12}),(C(\tilde{\rho},M))^{-2},
   (2C(\tilde{\rho}))^{-16},(2C(\tilde{\rho},M)(E_0+1))^{-2},(2C(\tilde{\rho})E_0^{\frac{1}{2}})^{-32},\right.\\
   &\qquad\quad\left. (2C(\tilde{\rho},M))^{-\frac{16}{9}},\left(\frac{\tilde{\rho}}{2C(\tilde{\rho},M)}\right)^{-12},(C(\tilde{\rho}))^{-1},\left(\frac{\tilde{\rho}}{4C(\tilde{\rho})(1+E_0)}\right)^3\right\}
   \end{aligned}
 \end{equation*}
 and
 \begin{equation*}
   \hat{\epsilon}=\min\left\{(2C)^{-\frac{1}{3}}, (27C(\Omega))^{-\frac{1}{3}}, (2CC(\Omega))^{-\frac{1}{3}}\right\}
 \end{equation*}
 with $C$ here depending only on $\mu,\lambda$ and $\Omega$, and $C(\Omega)$ only depending on $\Omega$.
 
 Moreover, if $\bar{\rho}\leq1$ and $\frac{\tilde{\rho}}{\bar{\rho}}\geq3$, then for any $\gamma\in(1,\frac{3}{2}]$, $r\in[1,\infty)$ and $p\in[1,6]$, there exists positive constants $\tilde{C}$ and $\eta_0$, with $\tilde{C}$ depending only on $\mu,\lambda,\gamma,a,\tilde{\rho},\bar{\rho},M,\Omega,r,p$ and the matrix $A$, and $\eta_0$ depending only on $\mu,\lambda,a,\Omega,r,p$, $\tilde{\rho}$ and $\frac{\tilde{\rho}}{\bar{\rho}}$, but independent of $\gamma-1$, such that for any $t\geq1$, it holds that
\begin{equation}\label{energy-decay-thm}
  \norm{\rho-\bar{\rho}}_{L^r}+\norm{u}_{W^{1,p}}+\norm{\sqrt{\rho}\dot{u}}_{L^2}\leq \tilde{C}e^{-\eta_0\bar{\rho}^{\gamma}t}.
\end{equation}
On the other hand, if $\gamma>\frac{3}{2}$, there exists similar exponent decay result as follows:
\begin{equation}\label{energy-decay-thm-large-gamma}
  \norm{\rho-\bar{\rho}}_{L^r}+\norm{u}_{W^{1,p}}+\norm{\sqrt{\rho}\dot{u}}_{L^2}\leq \tilde{C}_1e^{-\eta_1t}
\end{equation}
where positive constants $\tilde{C}_1$ and $\eta_1$ depend only on $\mu,\lambda,\gamma,a,\tilde{\rho},\bar{\rho},\Omega,r,p$ and the matrix $A$ (with $\tilde{C}_1$ depending on $M$ as well).
\end{thm}

With Theorem \ref{thm-global-CNS} in hand, we will give a corollary to state the large-time behavior of $\nabla\rho$ when vacuum appears initially. This corollary is a direct consequence from the exponent decay of classical solutions as in \eqref{energy-decay-thm} and \eqref{energy-decay-thm-large-gamma}, and the proof can be referred to \cite{Cai-Li2023} for details and is omitted here.

\begin{cor}
  Under the conditions of Theorem \ref{thm-global-CNS}, assume further that there exists certain point $x_0\in\Omega$ such that $\rho(x_0)=0$. Then the unique classical solution $(\rho, u)$ obtained in Theorem \ref{thm-global-CNS} satisfies that for any $r_1>3$, there exist positive constants $\tilde{C}_2$ and $\eta_2$, with $\tilde{C}_2$ depending only on $\mu,\lambda,\gamma,a,\tilde{\rho},\bar{\rho},\Omega,r_1$, and $\eta_2$ depending only on $\mu,\lambda,a,\Omega,r_1$, $\tilde{\rho}$ and $\frac{\tilde{\rho}}{\bar{\rho}}$, but independent of $\gamma-1$, such that for any $\gamma\in(1,\frac{3}{2}]$ and $t\geq1$,
  \begin{equation}\label{growth-cor1}
    \norm{\nabla\rho(t)}_{L^{r_1}}\geq\tilde{C}_2e^{\eta_2\bar{\rho}^{\gamma}t},
  \end{equation}
  and positive constants $\tilde{C}_3$ and $\eta_3$ depending only on $\mu,\lambda,\gamma,a,\tilde{\rho},\bar{\rho},\Omega,r_1$ such that for any $\gamma\in(\frac{3}{2},\infty)$ and $t\geq1$,
  \begin{equation}\label{growth-cor2}
    \norm{\nabla\rho(t)}_{L^{r_1}}\geq\tilde{C}_3e^{\eta_3t}.
  \end{equation}
\end{cor}

Here we list some remarks as follows.
\begin{rem}\label{remark-1}
  Compared to Hong-Hou-Peng-Zhu \cite{Zhu2024} where a Nishida-Smoller type large solution is obtained for the Cauchy problem of compressible Navier-Stokes equations, 
  the main novelties here can be outlined as follows.
  First, we need to deal with additional difficulties from boundary effects.
  Second, we prove that the classical solution has an exponential decay rate
which is decreasing with respect to the adiabatic exponent provided that the fluid is
nearly isothermal. This is totally new as compared to \cite{Zhu2024} where there is no information on the large time behavior of the solution.
Third, we show that the gradient of the
density will grow unboundedly with an exponential rate when the initial state contains vacuum.
(even at a point), which is also completely new as compared to \cite{Zhu2024}.
\end{rem}

\begin{rem}\label{remark-2}
  Compared to Cai-Li \cite{Cai-Li2023} where the global existence and large time behavior of classical solutions to \eqref{Large-CNS-eq}-\eqref{Navier-slip-condition} with
  small initial energy and vacuum are obtained, the main novelties can be outlined as follows.
  First, in our case, the initial energy $E_0$ is allowed to be large when $\gamma-1$ and the matrix $A$ are suitably small. Therefore, Theorem \ref{thm-global-CNS} is still applicable to the case that the initial energy $E_0$ is small for any given $\gamma$ and $A$. 
  Second, we give the explicit exponent decay rate presented in \eqref{energy-decay-thm} which is decreasing with respect to $\gamma$. This can be verified by $\bar{\rho}^{\gamma}\leq(a|\Omega|)^{-1}(\gamma-1)E_0$ in \eqref{bound-rho-average}. 
  It is worth mentioning that this phenomenon is totally new as compared to \cite{Cai-Li2023}.
\end{rem}

\begin{rem}\label{remark-3}
  Since our results allow large initial energy $E_0$ as $\gamma-1$ tends to 0, Theorem \ref{thm-global-CNS} can be viewed as a special extension of the uniqueness and regularity theory of weak solutions constructed by Lions \cite{Lions1998} and Feireisl \cite{Feireisl2001} which require that the initial energy is small, but allow large initial energy for $\gamma>\frac{3}{2}$. However, although the initial energy could be large as $\gamma$ close to 1, it is still open whether global classical solutions exist or not when the initial data are large for any given $\gamma$.  
\end{rem}

\begin{rem}\label{remark-4}
   In our results, we can extract that $(\gamma-1)E_0^{17}\leq C$, which means $E_0\leq C(\gamma-1)^{-\frac{1}{17}}$. This is very different from \cite{Zhu2017,Zhu2024} due to the slip boundary conditions \eqref{Navier-slip-condition}. This allows the large initial energy as $\gamma$ is close to 1. In addition, the smallness condition in \eqref{small-condition} imposed on the matrix $A$ is different from \cite{Zhu2024}, but it can be seen as a similar constraint on boundary as compared to the smallness assumption on far-field density in \cite{Zhu2024}. In particular, our results hold for the usual case that the matrix $A=0$. It should be mentioned that the smallness of $A$ only depends on $\mu,\lambda$ and $\Omega$, but independent of the density $\rho$, velocity $u$ and pressure $P$.
\end{rem}

\begin{rem}\label{remark-5}
  In addition to the conditions of Theorem \ref{thm-global-CNS}, if assuming further that $\norm{u_0}_{H^{\beta}}\leq\tilde{M}$ with $\beta\in(\frac{1}{2},1]$ instead of $\norm{\nabla u_0}_{L^2}\leq M$, then the conclusions in Theorem \ref{thm-global-CNS} still hold. This can be achieved by a similar way as in \cite{Cai-Li2023}. In our results, we also do not focus on the regularity of the bounded domain $\Omega$ and the matrix $A$, but we can make analogous discussions as in \cite{Cai-Li2023}.
\end{rem}

Now, we make some comments on the analysis of this paper. Similar to the arguments in \cite{Cai-Li2023} and \cite{Zhu2024}, the key issue in our proof is to derive the time-independent upper bound on the density $\rho$ (see Lemma \ref{lem-rho-bound}). Compared to \cite{Cai-Li2023} where the analysis relies heavily on the smallness of the initial energy $E_0$,
the new difficulty here lies in that in our case, the initial energy could be large when the adiabatic exponent $\gamma$ is sufficiently close to 1.
Indeed, with the help of the smallness of the initial energy $E_0$, Cai-Li \cite{Cai-Li2023} derives the smallness of
$\displaystyle\int_0^T\norm{\nabla u}_{L^2}^2$ and $\displaystyle\norm{P}_{L^1}$ by the elementary energy estimate directly (see Lemma \ref{lem-essential-energy}) . 
However, due to the absence of smallness on the initial energy, we can not extract the smallness of 
$\displaystyle\int_0^T\norm{\nabla u}_{L^2}^2$ from the elementary energy estimate.
Therefore, we need develop some new ingredients to overcome this difficulty.
On the other hand, as compared to Hong-Hou-Peng-Zhu \cite{Zhu2024} where the Cauchy problem is considered, we need employ some new observations and ideas to overcome the difficulties from the Navier-slip boundary conditions \eqref{Navier-slip-condition}.  We now highlight the main
differences and ingredients as follows:
\begin{itemize}
  \item As mentioned before, to derive the time-independent upper bound of the density $\rho$, due to the lack of smallness of $E_0$, the smallness of $\displaystyle\int_0^T\norm{\nabla u}_{L^2}^2$ can not be extracted from the basic energy estimate directly. However, we can obtain 
      the smallness of $\norm{P}_{L^1}$ by the basic energy estimate directly. This is very different from \cite{Zhu2024} where the proof depends heavily on the relationship between $\norm{P(\rho)-P(\tilde{\rho})}_{L^2}^2$ and $\displaystyle\int G(\rho)$.
      It should be mentioned that we give an explicit relationship of these two terms  in Remark \ref{rem-relationship} by employing careful analysis. 
      
  \item Since the initial energy $E_0$ could be large in our analysis, we can only get the smallness of $\displaystyle\int_0^{\sigma(T)}\norm{\nabla u}_{L^2}^2$ rather than $\displaystyle\int_0^T\norm{\nabla u}_{L^2}^2$ (see Lemma \ref{lem-essential-estimate}). By making delicate energy estimates, we can get the estimates of $A_1(T)$ and $A_2(T)$ stated in Lemma \ref{lem-A1-A2-control}:
      \begin{equation}\label{A1-A2-inequality1}\begin{aligned}
        A_2(T)&\leq C A_1(\sigma(T))+CA_1^{\frac{3}{2}}(T)+C(\tilde{\rho})A_1^3(E_0+1)\\
        &\qquad+\cdots+\int_0^T\sigma^3(\norm{\nabla u}_{L^4}^4+\norm{P-\bar{P}}_{L^4}^4),
     \end{aligned} \end{equation}
      \begin{equation}\label{A1-A2-inequality2}
        \begin{aligned}
        A_1(T)&\leq A_1(\sigma(T))+C\int_{\sigma(T)}^T\int\sigma(|P-\bar{P}||\nabla u|^2+|\nabla u|^3)+C(\tilde{\rho})\int_{\sigma(T)}^T\sigma\norm{\nabla u}_{L^2}^4\\
        &\leq A_1(\sigma(T))+C(\tilde{\rho})\left(\int_{\sigma(T)}^T\norm{\nabla u}_{L^2}^2\right)^{\frac{1}{2}}A_1^{\frac{3}{4}}(T)A_2^{\frac{1}{4}}(T)\\
        &\leq A_1(\sigma(T))+C(\tilde{\rho})E_0^{\frac{1}{2}}A_1^{\frac{3}{4}}(T)A_2^{\frac{1}{4}}(T),\quad(\textrm{see \eqref{A1-estimate-3}-\eqref{A1-estimate-5}})
        \end{aligned}
      \end{equation}
      which is new and very different from those in \cite{Cai-Li2023, Zhu2024}. Indeed, by virtue of the Navier-slip boundary conditions \eqref{Navier-slip-condition}, the term $A_1^{\frac{3}{2}}(T)$ in \eqref{A1-A2-inequality1} is caused by boundary term $\displaystyle\int_{\pa\Omega}\sigma^mu\cdot\nabla n\cdot uG$ (see \eqref{A-control-3}) which is inevitable.
      This term combined with the second estimate \eqref{A1-A2-inequality2} implies that the order of $A_2(T)$ should be higher than that of $A_1(T)$, but lower than that of $A_1^{\frac{3}{2}}(T)$. Based on this key observation, we specifically choose $A_2(T)\sim A_1(T)^{\frac{4}{3}}$ in Proposition \ref{prop-a-priori-estimate}. It is worth mentioning that this is different from \cite{Zhu2024} where $A_1(T)\sim A_2(T)^2$ is chosen.
      In addition, due to lack of the smallness of $E_0$, we have to derive new estimates on almost all boundary terms and intermediate terms appearing in the control of $A_2(T)$, such as $\displaystyle\int_{\pa\Omega}G(u\cdot\nabla u)\cdot\nabla n\cdot u$ in \eqref{boundary-term-A2-2} and $\norm{\nabla u}_{L^2}^2\norm{\nabla G}_{L^3}^2$ in \eqref{A2-estimate-transition1}, which is very different from \cite{Cai-Li2023}.
  \item The control of $\displaystyle\int_0^T\sigma^3(\norm{\nabla u}_{L^4}^4+\norm{P-\bar{P}}_{L^4}^4)$ appearing in \eqref{A1-A2-inequality1} is the most difficult part of this paper. By applying Lemma \ref{lem-curl-effective-viscous} or using the similar procedure to deal with $\displaystyle\int_0^T\sigma^3\norm{P-\bar{P}}_{L^4}^4$  as in Lemma \ref{lem-control-A1-A2} will yield the same trouble that
      \begin{equation}\label{trouble-estimate1}
        \int_0^T\sigma^3\norm{P-\bar{P}}_{L^4}^4\leq C\int_0^T\sigma^3\norm{\nabla u}_{L^2}^4+\hbox{good terms},
      \end{equation}
      \begin{equation}\label{trouble-estimate2}
        \int_0^T\sigma^3\norm{\nabla u}_{L^4}^4\leq C\int_0^T\sigma^3(\norm{\nabla u}_{L^2}^4+\norm{P-\bar{P}}_{L^4}^4)+\hbox{good terms},
      \end{equation}
      where the trouble term $\displaystyle\int_0^T\sigma^3\norm{\nabla u}_{L^2}^4$  in \eqref{trouble-estimate1} and \eqref{trouble-estimate2} is out of control due to lack of  smallness on  $\displaystyle\int_{\sigma(T)}^T\sigma^3\norm{\nabla u}_{L^2}^4$ (smaller than $A_2(T)$). 
      Noticing that for the Cauchy problem in \cite{Zhu2024}, the trouble term $\displaystyle\int_0^T\sigma^3\norm{\nabla u}_{L^2}^4$ in \eqref{trouble-estimate1} and \eqref{trouble-estimate2}
      exactly disappears. Then, by combining the relations \eqref{trouble-estimate1} and \eqref{trouble-estimate2}, the authors in \cite{Zhu2024} succeed in controlling 
      $\displaystyle\int_0^T\sigma^3\norm{P-\bar{P}}_{L^4}^4$ and hence $\displaystyle\int_0^T\sigma^3\norm{\nabla u}_{L^4}^4$.
      Therefore, we need employ some new thoughts to overcome this difficulty. 
      Fortunately, we observe that the term $\displaystyle\int_{\sigma(T)}^T\sigma^3\norm{\nabla u}_{L^2}^4$ is actually from $\curl u\times n|_{\pa\Omega}=-Au$ of Navier-slip boundary conditions \eqref{Navier-slip-condition}. If some smallness condition is imposed on the matrix $A$, then we can close the estimate of $\displaystyle\int_{\sigma(T)}^T\sigma^3\norm{\nabla u}_{L^4}^4$ and hence that of $\displaystyle\int_0^T\sigma^3\norm{P-\bar{P}}_{L^4}^4$. This part is discussed in Remark \ref{rem-small-A}, \eqref{P-L^4-estimate2}, \eqref{A2-estimate-2} and \eqref{A2-estimate-3}. It is worth mentioning that the smallness condition in \eqref{small-condition} imposed on the matrix $A$ is different from \cite{Zhu2024}, but it can be seen as a similar constraint on boundary as compared to the smallness assumption on far-field density in \cite{Zhu2024}. In particular, our results hold for the usual case that the matrix $A=0$. Additionally, we should remark that the smallness of $A$ only depends on $\mu,\lambda$ and $\Omega$, but independent of the density $\rho$, velocity $u$ and pressure $P$.
  \item The next important issue is to give the exponent decay of the classical solution obtained in Theorem \ref{thm-global-CNS}. Since we aim to describe the monotonicity of exponent decay rate with respect to the adiabatic exponent $\gamma$, the rough relation among $(\rho-\bar{\rho}),G(\rho)$ and $(P(\rho)-P(\bar{\rho}))(\rho-\bar{\rho})$ in \cite{Cai-Li2023} is not applicable for our case. To overcome this difficulty, by 
       employing several key observations, we give an explicit relationship among three terms above in Lemma \ref{lem-relationship}. With the crucial  Lemma \ref{lem-relationship} in hand,  we succeed in deriving the exponent decay rate for any $\gamma\in(1,\frac{3}{2}]$ in Proposition \ref{prop-energy-decay}.
\end{itemize}

The rest of the paper is organized as follows: In the next section, we introduce some elementary lemmas that will be needed later. In Section \ref{sect-proof-thm}, we give the proof of Theorem \ref{thm-global-CNS}.

\section{Preliminary}
This section mainly introduces some elementary lemmas used later. First, we give the local existence of strong solutions as follows. 

\begin{lem}\label{lem-local-sol}
  Let $\Omega$ be as in Theorem \ref{thm-global-CNS}, and assume that $(\rho_0,u_0)$ satisfies \eqref{initial-data1}-\eqref{compatibility-condition}. Then there exist a small $T>0$ and a unique strong solution $(\rho,u)$ to the problem \eqref{Large-CNS-eq}-\eqref{Navier-slip-condition} on $\Omega\times(0,T]$ satisfying for any $\tau\in(0,T)$,
  \begin{equation*}
    \begin{cases}
      (\rho,P)\in C([0,T];W^{2,q}),\\
     \nabla u\in C([0,T];H^1)\cap L^{\infty}(\tau,T;W^{2,q}),\\
     u_t\in L^{\infty}(\tau,T;H^2)\cap H^1(\tau,T;H^1),\\
     \sqrt{\rho}u_t\in L^{\infty}(0,\infty;L^2).
    \end{cases}
  \end{equation*}
\end{lem}
This lemma can be deduced by combining the local existence result in \cite{Huang2021} and the initial-boundary-value problem under Navier boundary conditions without vacuum in \cite{Hoff2012}.

Next, the well-known Poincar\'{e} inequality and Gagliardo-Nirenberg interpolation inequality will be used frequently later.

\begin{lem}\label{lem-Poincare-GN}
Assume that $\Omega$ is a bounded domain in $\mathbb{R}^3$ with Lipschitz boundary. Then\\ 
(1) For any $p\in(1,\infty)$, there exists a constant $C>0$ such that (see \cite{Berselli-Spirito2012})
  \begin{equation}\label{Poincare-inequality}
    \norm{f}_{L^p}\leq C\norm{\nabla f}_{L^p}, \textrm{ if $f\cdot n|_{\pa\Omega}=0$ or $f\times n|_{\pa\Omega}=0$ or if $\bar{f}=0$ and $\Omega$ is connected}.
  \end{equation}
(2) For $p\in[2,6]$, $q\in(1,\infty)$, and $r\in(3,\infty)$, there exist generic constants $C_i>0(i=1,\dots,4)$ which depend only on $p,q,r$ and $\Omega$ such that (see \cite{Nirenberg1959})
\begin{equation}\label{GN-inequality1}
  \norm{f}_{L^p}\leq C_1\norm{f}_{L^2}^{\frac{6-p}{2p}}\norm{\nabla f}_{L^2}^{\frac{3p-6}{2p}}+C_2\norm{f}_{L^2},
\end{equation}
\begin{equation}\label{GN-inequality2}
  \norm{g}_{C(\bar{\Omega})}\leq C_3\norm{g}_{L^q}^{\frac{q(r-3)}{3r+q(r-3)}}\norm{\nabla g}_{L^r}^{\frac{3r}{3r+q(r-3)}}+C_4\norm{g}_{L^2}.
\end{equation}
In particular, if either $f\cdot n|_{\pa\Omega}=0$ or $\bar{f}=0$ with $\Omega$ connected, we can set $C_2=0$. Similarly, the constant $C_4=0$ if $g\cdot n|_{\pa\Omega}=0$ or if $\bar{g}=0$ with $\Omega$ connected.
\end{lem}

The following Zlotnik's inequality is introduced to get the upper bound of the density $\rho$.

\begin{lem}[see \cite{Zlotnik2000}]\label{lem-Zlotnik-inequality}
  Suppose the function $y$ satisfies that
  \begin{equation*}
    y'(t)=g(y)+b'(t),\, t\in[0,T], \quad y(0)=y_0,
  \end{equation*}
  with $g\in C(\mathbb{R})$ and $y,b\in W^{1,1}(0,T)$. If $g(\infty)=-\infty$ and
  \begin{equation}\label{Zlotnik-condition1}
    b(t_2)-b(t_1)\leq N_0+N_1(t_2-t_1)
  \end{equation}
  for all $0\leq t_1<t_2\leq T$ with some $N_0\geq 0$ and $N_1\geq 0$, then
  \begin{equation*}
    y(t)\leq\max\{y_0,\zeta_0\}+N_0<\infty \textrm{ on }[0,T],
  \end{equation*}
  where $\zeta_0$ is a constant such that
  \begin{equation}\label{Zlotnik-condition2}
    g(\zeta)\leq -N_1 \textrm{ for }\zeta\geq\zeta_0.
  \end{equation}
\end{lem}

Next, we consider the Lam\'{e}'s system
\begin{equation}\label{Lame-system}
  \begin{cases}
     -\mu\Delta u-(\lambda+\mu)\nabla\div u=f,\text{ in }\Omega,\\
     u\cdot n=0,\, \curl u\times n=-Au,\text{ on }\pa\Omega.
  \end{cases}
\end{equation}
This system is the strongly elliptic equations, and thus the standard elliptic estimates hold as follows.

\begin{lem}[see \cite{Agmon-Douglis-Nirenberg1964}]\label{lem-Lame-estimate}
Let $u$ be a smooth solution of the Lam\'{e}'s system \eqref{Lame-system}. Then for $p\in(1,\infty)$ and integer $k\geq0$, there exists a constant $C>0$ depending only on $\lambda,\mu,p,k,\Omega$ and the matrix $A$ such that
\begin{itemize}
  \item If $f\in W^{k,p}$, then
  \begin{equation}\label{Lame-estimate1}
    \norm{u}_{W^{k+2,p}}\leq C(\norm{f}_{W^{k,p}}+\norm{u}_{L^p});
  \end{equation}
  \item If $f=\nabla g$ and $g\in W^{k,p}$, then
  \begin{equation}\label{Lame-estimate2}
    \norm{u}_{W^{k+1,p}}\leq C(\norm{g}_{W^{k,p}}+\norm{u}_{L^p}).
  \end{equation}
\end{itemize} 
\end{lem}

Next, the following two Hodge-type decompositions are given in \cite{Aramaki2014,von-Wahl1992}.

\begin{lem}\label{lem-Hodge-decomposition}
  Let integer $k\geq 0$ and $p\in(1,\infty)$, and assume that $\Omega$ is a simply connected bounded domain in $\mathbb{R}^3$ with $C^{k+1,1}$ boundary $\pa\Omega$. Then there exists a constant $C=C(p,k,\Omega)>0$ such that
  \begin{itemize}
    \item If $v\in W^{k+1,p}$ with $v\cdot n|_{\pa\Omega}=0$,
    \begin{equation}\label{Hodge-decomposition1}
      \norm{v}_{W^{k+1,q}}\leq C(\norm{\div v}_{W^{k,p}}+\norm{\curl v}_{W^{k,p}}).
    \end{equation}
    In particular, for $k=0$, we have
    \begin{equation*}
      \norm{\nabla v}_{L^p}\leq C(\norm{\div v}_{L^p}+\norm{\curl v}_{L^p}).
    \end{equation*}
    \item If the boundary $\pa\Omega$ only has a finite number of 2-dimensional connected components and $v\in W^{k+1,p}$ with $v\times n|_{\pa\Omega}=0$, then
        \begin{equation}\label{Hodge-decomposition2}
      \norm{v}_{W^{k+1,q}}\leq C(\norm{\div v}_{W^{k,p}}+\norm{\curl v}_{W^{k,p}}+\norm{v}_{L^p}).
    \end{equation}
    In particular, if $\Omega$ has no holes, then
    \begin{equation}\label{Hodge-decomposition3}
      \norm{v}_{W^{k+1,q}}\leq C(\norm{\div v}_{W^{k,p}}+\norm{\curl v}_{W^{k,p}}).
    \end{equation}
  \end{itemize} 
\end{lem}

Next, we introduce the Bogovskii operator $\mathcal{B}$ which solves the following problem
\begin{equation}\label{div-equation}
  \begin{cases}
    \div\mathcal{B}[f]=f,\text{ in }\Omega, \\
    \mathcal{B}[f]=0,\text{ on }\pa\Omega.
  \end{cases}
\end{equation}
Then, we have the following conclusion.
\begin{lem}[see \cite{Galdi2011}]\label{lem-Bogovskii-operator}
 There exists a linear operator $\mathcal{B}(f): \{f\in L^p: \bar{f}=0\}\longrightarrow\mathbb{R}^3$ solving the problem \eqref{div-equation} and satisfying that for any $p\in(1,\infty)$,
 \begin{equation}\label{Bogovskii-estimate1}
   \norm{\mathcal{B}(f)}_{W_0^{1,p}(\Omega)}\leq C(p)\norm{f}_{L^p(\Omega)}.
 \end{equation}
 In particular, if $f=\div g$ with $g\in L^r(\Omega)$ and $g\cdot n|_{\Omega}=0$, then
 \begin{equation}\label{Bogovskii-estimate2}
   \norm{\mathcal{B}(f)}_{L^r(\Omega)}\leq C(r)\norm{g}_{L^r(\Omega)},\,\forall r\in(1,\infty).
 \end{equation}
\end{lem}

Next, we write $\eqref{Large-CNS-eq}_2$ as
\begin{equation}\label{momentum-equation}
  \rho\dot{u}=\nabla G-\mu\nabla\times\curl u
\end{equation}
with 
\begin{equation*}
  \curl u=\nabla\times u,\, G=(2\mu+\lambda)\div u-(P-\bar{P}),\, \dot{f}:=f_t+u\cdot\nabla f,
\end{equation*}
where the vorticity $\curl u$ and the effective viscous flux $G$ both play an important role in the following analysis. Here, we give the following key a priori estimates on $\curl u$ and $G$ which will be used frequently.

\begin{lem}\label{lem-curl-effective-viscous}
  Assume $\Omega$ is a simply connected bounded domain in $\mathbb{R}^3$ and its smooth boundary $\pa\Omega$ only has a finite number of 2D connected components. Let $(\rho,u)$ be a smooth solution of \eqref{Large-CNS-eq} with Navier-slip boundary conditions \eqref{Navier-slip-condition}. Then for any $p\in[2,6]$ and $q\in(1,\infty)$, there exist constants $C_1,\, C>0$ depending only on $p,q,\mu,\lambda$ and $\Omega$ (with $C$ depending on $A$ as well) such that
  \begin{equation}\label{curl-effective-viscous-estimate1}
    \norm{\nabla u}_{L^q}\leq C_1(\norm{\div u}_{L^q}+\norm{\curl u}_{L^q}),
  \end{equation}
  \begin{equation}\label{curl-effective-viscous-estimate2}
    \norm{\nabla G}_{L^p}+\norm{\nabla\curl u}_{L^p}\leq C(\norm{\rho\dot{u}}_{L^p}+\norm{\nabla u}_{L^2}+\norm{P-\bar{P}}_{L^p}),
  \end{equation}
  \begin{equation}\label{curl-effective-viscous-estimate3}
    \norm{G}_{L^p}\leq C\norm{\rho\dot{u}}_{L^2}^{\frac{3p-6}{2p}}(\norm{\nabla u}_{L^2}+\norm{P-\bar{P}}_{L^2})^{\frac{6-p}{2p}}+C(\norm{\nabla u}_{L^2}+\norm{P-\bar{P}}_{L^2}),
  \end{equation}
  \begin{equation}\label{curl-effective-viscous-estimate4}
    \norm{\curl u}_{L^p}\leq C\norm{\rho\dot{u}}_{L^2}^{\frac{3p-6}{2p}}\norm{\nabla u}_{L^2}^{\frac{6-p}{2p}}+C\norm{\nabla u}_{L^2}.
  \end{equation}
  Moreover,
  \begin{equation}\label{curl-effective-viscous-estimate5}
    \norm{G}_{L^p}+\norm{\curl u}_{L^p}\leq C(\norm{\rho\dot{u}}_{L^2}+\norm{\nabla u}_{L^2}),
  \end{equation}
  \begin{equation}\label{curl-effective-viscous-estimate6}
    \norm{\nabla u}_{L^p}\leq C\norm{\rho\dot{u}}_{L^2}^{\frac{3p-6}{2p}}(\norm{\nabla u}_{L^2}+\norm{P-\bar{P}}_{L^2})^{\frac{6-p}{2p}}+C(\norm{\nabla u}_{L^2}+\norm{P-\bar{P}}_{L^p}).
  \end{equation}
  In particular, for $p=2$, the term $\norm{P-\bar{P}}_{L^p}$ can be removed in \eqref{curl-effective-viscous-estimate2}.  
\end{lem}

\begin{proof}
  Since the proof of this lemma is similar to that of \cite{Cai-Li2023}, we only give a sketch for simplicity. First, the inequality \eqref{curl-effective-viscous-estimate1} just follows the Hodge-type decomposition in Lemma \eqref{lem-Hodge-decomposition}. For the estimate on $\nabla F$, we consider the following elliptic equations:
  \begin{equation}\label{G-elliptic-equation}
    \begin{cases}
      \Delta G=\div(\rho\dot{u}), & x\in\Omega, \\
      \frac{\pa G}{\pa n}=(\rho\dot{u}-\mu\nabla\times(Au)^{\bot})\cdot n, & x\in\pa\Omega.
    \end{cases}
  \end{equation}
  with the notation
  \begin{equation}\label{defi-vertical}
    f^{\bot}=-f\times n=n\times f.
  \end{equation}
  Here, we should notice that the normal vector $n$ only makes sense on boundary $\pa\Omega$, and can be extended to a smooth vector-valued function on $\bar{\Omega}$. Thus $f^{\bot}$ is well-defined on $\bar{\Omega}$. Similar arguments can be also applicable to $(Au)^{\bot}$.
  
  Due to \eqref{Navier-slip-condition}, $(\curl u+(Au)^{\bot})\times n=0$ on $\Omega$. Then we have for any $\eta\in C^{\infty}(\bar{\Omega})$
  \begin{equation*}
    \begin{aligned}
    \int\nabla\times\curl u\cdot\nabla\eta&=\int(\nabla\times(\curl u+(Au)^{\bot})\cdot\nabla\eta-\int\nabla\times(Au)^{\bot}\cdot\nabla\eta\\
    &=-\int\nabla\times(Au)^{\bot}\cdot\nabla\eta,
    \end{aligned}
  \end{equation*}
  which combined with \eqref{momentum-equation} implies that
  \begin{equation*}
    \int\nabla G\cdot\nabla\eta=\int(\rho\dot{u}-\mu\nabla\times(Au)^{\bot})\cdot\nabla\eta,\quad \forall\eta\in C^{\infty}(\bar{\Omega}).
  \end{equation*}
  Then the standard elliptic estimate on \eqref{G-elliptic-equation} yields that for any $q\in(1,\infty)$
  \begin{equation}\label{G-transition-estimate1}
    \norm{\nabla G}_{L^q}\leq C\norm{\rho\dot{u}-\mu\nabla\times(Au)^{\bot}}_{L^q}\leq C(\norm{\rho\dot{u}}_{L^q}+\norm{\nabla u}_{L^q}),
  \end{equation}
  where we have applied the Poincar\'{e}'s inequality \eqref{Poincare-inequality}, and for any integer $k\geq0$,
  \begin{equation}\label{G-transition-estimate2}
    \norm{\nabla G}_{W^{k+1,q}}\leq C(\norm{\rho\dot{u}}_{W^{k,q}}+\norm{\nabla u}_{W^{k,q}}).
  \end{equation}
  For the vorticity $\curl u$, due to $(\curl u+(Au)^{\bot})\times n|_{\pa\Omega}=0$ and \eqref{momentum-equation}, by Lemma \ref{lem-Hodge-decomposition}, we get that for any $q\in(1,\infty)$
  \begin{equation}\label{curl-transition-estimate1}
  \begin{aligned}
    \norm{\nabla\curl u}_{L^q}&\leq C(\norm{\nabla\times\curl u}_{L^q}+\norm{\nabla(Au)^{\bot}}_{L^q})\\
    &\leq C(\norm{\rho\dot{u}}_{L^q}+\norm{\nabla G}_{L^q}+\norm{\nabla u}_{L^q})\\
    &\leq C(\norm{\rho\dot{u}}_{L^q}+\norm{\nabla u}_{L^q}),
  \end{aligned}
  \end{equation}
  and for any integer $k\geq 0$,
  \begin{equation}\label{curl-transition-estimate2}
    \begin{aligned}
    \norm{\nabla\curl u}_{W^{k+1,q}}&\leq C(\norm{\nabla\times\curl u}_{W^{k+1,q}}+\norm{\curl u}_{L^q}+\norm{(Au)^{\bot}}_{W^{k+2,q}})\\
    &\leq C(\norm{\rho\dot{u}}_{W^{k+1,q}}+\norm{\nabla(Au)^{\bot}}_{W^{k+1,q}}+\norm{\nabla u}_{L^q}).
    \end{aligned}
  \end{equation}
  Since $\bar{F}=0$, one can deduce \eqref{curl-effective-viscous-estimate3} and \eqref{curl-effective-viscous-estimate4} from Gagliardo-Nirenberg interpolation inequality \eqref{GN-inequality1}, \eqref{G-transition-estimate1} and \eqref{curl-transition-estimate1}. The inequality \eqref{curl-effective-viscous-estimate5} follows from Sobolev imbedding, Poincar\'{e}'s inequality \eqref{Poincare-inequality}, \eqref{G-transition-estimate1} and \eqref{curl-transition-estimate1} directly. The last inequality \eqref{curl-effective-viscous-estimate6} is a combination of \eqref{curl-effective-viscous-estimate1}, \eqref{curl-effective-viscous-estimate3} and \eqref{curl-effective-viscous-estimate4}. The second inequality \eqref{curl-effective-viscous-estimate2} follows from \eqref{curl-effective-viscous-estimate6}, \eqref{G-transition-estimate1} and \eqref{curl-transition-estimate1}.
\end{proof}

\begin{rem}\label{rem-small-A}
  In fact, we may need some refined inequalities to deal with $\displaystyle\int_0^T\sigma^3\norm{\nabla u}_{L^4}^4$ in Lemma \ref{lem-control-A1-A2}. Here we have the modified estimates as follows:
  \begin{equation}\label{curl-G-modified-estimate1}
    \norm{\nabla G}_{L^2}+\norm{\nabla\curl u}_{L^2}\leq C(\norm{\rho\dot{u}}_{L^2}+\norm{\nabla(Au)^{\bot}}_{L^2})\leq C(\norm{\rho\dot{u}}_{L^2}+\norm{A}_{W^{1,6}}\norm{\nabla u}_{L^2}),
  \end{equation}
  and due to $(\curl u+(Au)^{\bot})\times n|_{\pa\Omega}=0$ and Poincar\'{e} inequality \eqref{Poincare-inequality},
  \begin{equation}\label{curl-G-modified-estimate2}
    \begin{aligned}
    \norm{\curl u}_{L^2}&\leq C\norm{\nabla(\curl u+(Au)^{\bot})}_{L^2}+\norm{(Au)^{\bot}}_{L^2}\\
    &\leq C(\norm{\nabla\curl u}_{L^2}+\norm{\nabla(Au)^{\bot}}_{L^2})\\
    &\leq C(\norm{\rho\dot{u}}_{L^2}+\norm{A}_{W^{1,6}}\norm{\nabla u}_{L^2}),
    \end{aligned}
  \end{equation}
  and henceforth
  \begin{equation}\label{curl-G-modified-estimate3}
    \begin{aligned}
      \norm{G}_{L^6}+\norm{\curl u}_{L^6}&\leq C(\norm{\nabla G}_{L^2}+\norm{\nabla\curl u}_{L^2}+\norm{\curl u}_{L^2})\\
      &\leq C(\norm{\rho\dot{u}}_{L^2}+\norm{A}_{W^{1,6}}\norm{\nabla u}_{L^2}).
    \end{aligned}
  \end{equation}
  Now we give the estimate on $\norm{\nabla u}_{L^4}^4$ as
  \begin{equation}\label{nabla-u-L^4-estimate1}
    \begin{aligned}
    \norm{\nabla u}_{L^4}^4&\leq C(\norm{G}_{L^4}^4+\norm{\curl u}_{L^4}^4+\norm{P-\bar{P}}_{L^4}^4)\\
    &\leq C(\norm{G}_{L^2}\norm{G}_{L^6}^3+\norm{\curl u}_{L^2}\norm{\curl u}_{L^6}^3)+C\norm{P-\bar{P}}_{L^4}^4\\
    &\leq C(\norm{\nabla u}_{L^2}+\norm{P-\bar{P}}_{L^2})(\norm{\rho\dot{u}}_{L^2}^3+\norm{A}_{W^{1,6}}^3\norm{\nabla u}_{L^2}^3)+C\norm{P-\bar{P}}_{L^4}^4\\
    &\leq C\norm{A}_{W^{1,6}}^3\norm{\nabla u}_{L^4}^4+C(\norm{\nabla u}_{L^2}+\norm{P-\bar{P}}_{L^2})\norm{\rho\dot{u}}_{L^2}^3+C(\norm{A}_{W^{1,6}}^3+1)\norm{P-\bar{P}}_{L^4}^4, 
    \end{aligned}
  \end{equation}
  which implies that
  \begin{equation}\label{nabla-u-L^4-estimate2}
    \norm{\nabla u}_{L^4}^4\leq C(\norm{\nabla u}_{L^2}+\norm{P-\bar{P}}_{L^2})\norm{\rho\dot{u}}_{L^2}^3+C\norm{P-\bar{P}}_{L^4}^4
  \end{equation}
  provided
  \begin{equation}\label{small-A-condition}
    C\norm{A}_{W^{1,6}}^3\leq\frac{1}{2},\quad\textrm{i.e.},\quad \norm{A}_{W^{1,6}}\leq (2C)^{-\frac{1}{3}}
  \end{equation}
  with $C$ only depending on $\mu,\lambda$ and $\Omega$.
\end{rem}

Next, we give the a priori estimate on $\dot{u}$ that will be used later. The detailed proof of following lemma also can be found in \cite{Cai-Li2023} and here we still present a sketch.

\begin{lem}\label{lem-dot-u}
  Let $(\rho,u)$ be a smooth solution of \eqref{Large-CNS-eq} with Navier-slip boundary conditions \eqref{Navier-slip-condition}. Then there exists a constant $C>0$ depending only on $\Omega$ such that
  \begin{equation}\label{dot-u-estimate1}
    \norm{\dot{u}}_{L^6}\leq C(\norm{\nabla\dot{u}}_{L^2}+\norm{\nabla u}_{L^2}^2),
  \end{equation}
  \begin{equation}\label{dot-u-estimate2}
    \norm{\nabla\dot{u}}_{L^2}\leq C(\norm{\div\dot{u}}_{L^2}+\norm{\curl\dot{u}}_{L^2}+\norm{\nabla u}_{L^4}^2).
  \end{equation}
\end{lem}

\begin{proof}
  Since it holds that
  \begin{equation*}
    \dot{u}\cdot n=u\cdot\nabla u\cdot n=-u\cdot\nabla n\cdot u=-u\cdot\nabla n\cdot(u^{\bot}\times n)=-(u\cdot\nabla n)\times u^{\bot}\cdot n \textrm{ on }\pa\Omega
  \end{equation*}
 from the simple fact that
  \begin{equation*}
    (a\times b)\cdot c=(b\times c)\cdot a=(c\times a)\cdot b,
  \end{equation*}
  we have
  \begin{equation}\label{dot-u-transition1}
    (\dot{u}+(u\cdot\nabla n)\times u^{\bot})\cdot n=0 \textrm{ on }\pa\Omega.
  \end{equation}
  This together with Sobolev imbedding and Poincar\'{e}'s inequality \eqref{Poincare-inequality} yields that
  \begin{equation*}
  \begin{aligned}
    \norm{\dot{u}}_{L^6}&\leq C(\norm{\nabla\dot{u}}_{L^2}+\norm{\bar{\dot{u}}}_{L^6})\leq C(\norm{\nabla\dot{u}}_{L^2}+\norm{\dot{u}}_{L^{\frac{3}{2}}})\\
    &\leq C(\norm{\nabla\dot{u}}_{L^2}+\norm{\nabla(\dot{u}+(u\cdot\nabla n)\times u^{\bot})}_{L^{\frac{3}{2}}}+\norm{(u\cdot\nabla n)\times u^{\bot}}_{L^{\frac{3}{2}}})\\
    &\leq C(\norm{\nabla\dot{u}}_{L^2}+\norm{\nabla u}_{L^2}^2).
  \end{aligned}
  \end{equation*}
  Finally combining \eqref{dot-u-transition1} with \eqref{Hodge-decomposition1} and Poincar\'{e}'s inequality \eqref{Poincare-inequality}, we deduce
  \begin{equation*}
    \begin{aligned}
    \norm{\nabla\dot{u}}_{L^2}&\leq C(\norm{\div\dot{u}}_{L^2}+\norm{\curl\dot{u}}_{L^2}+\norm{\nabla((u\cdot\nabla n)\times u^{\bot})}_{L^2})\\
    &\leq C(\norm{\div\dot{u}}_{L^2}+\norm{\curl\dot{u}}_{L^2}+\norm{\nabla u}_{L^4}^2).
    \end{aligned}
  \end{equation*}
\end{proof}

\section{Proof of Theorem \ref{thm-global-CNS}}\label{sect-proof-thm}

\subsection{Lower-order a priori estimates}
In this subsection, we are devoted to establishing some necessary a priori estimates for smooth solution $(\rho,u)$ to the problem \eqref{Large-CNS-eq}-\eqref{Navier-slip-condition} on $\Omega\times(0,T]$ for some fixed time $T>0$. Setting $\sigma=\sigma(t)=\min\{1,t\}$, we define
\begin{equation}\label{defi-a-priori-estimate}
  \begin{cases}
    A_1(T)=\sup\limits_{t\in[0,T]}\sigma\displaystyle\int|\nabla u|^2+\int_0^T\displaystyle\int\sigma\rho|\dot{u}|^2, \\
    A_2(T)=\sup\limits_{t\in[0,T]}\sigma^3\displaystyle\int\rho|\dot{u}|^2+\int_0^T\displaystyle\int\sigma^3|\nabla\dot{u}|^2, \\
    A_3(T)=\sup\limits_{t\in[0,T]}\displaystyle\int\rho|u|^3.
  \end{cases}
\end{equation}
Since for the large adiabatic exponent $\gamma>1$, the initial energy $E_0$ in \eqref{defi-total-energy} correspondingly becomes small from the smallness of $\mathcal{E}_0$ in \eqref{small-condition}, without of generality, we assume that
\begin{equation}\label{assume-condition1}
  \epsilon\leq 1, \quad 1<\gamma\leq\frac{3}{2}.
\end{equation}

Then, we give the following proposition that can guarantees the existence of a global classical solution of \eqref{Large-CNS-eq}-\eqref{Navier-slip-condition}.

\begin{pro}\label{prop-a-priori-estimate}
  Assume that the initial data satisfy \eqref{initial-data1}, \eqref{initial-data2} and \eqref{compatibility-condition}. If the solution $(\rho,u)$ satisfy
  \begin{equation}\label{a-priori-assumption}
    A_1(T)\leq2\mathcal{E}_0^{\frac{3}{8}},\quad A_2(T)\leq2\mathcal{E}_0^{\frac{1}{2}},\quad A_3(\sigma(T))\leq2\mathcal{E}_0^{\frac{1}{4}},\quad 0\leq\rho\leq2\tilde{\rho},
  \end{equation}
  then the following estimates hold:
  \begin{equation}\label{a-priori-goal}
    A_1(T)\leq\mathcal{E}_0^{\frac{3}{8}},\quad A_2(T)\leq\mathcal{E}_0^{\frac{1}{2}},\quad A_3(\sigma(T))\leq\mathcal{E}_0^{\frac{1}{4}},\quad 0\leq\rho\leq\frac{7}{4}\tilde{\rho},
  \end{equation}
  provided $\mathcal{E}_0\leq\epsilon$ and $\norm{A}_{W^{1,6}}\leq\hat{\epsilon}$, where $\epsilon>0$ is a small constant depending on $\mu, \lambda, a, \tilde{\rho}, \Omega, M, E_0$, and the matrix $A$, but independent of $\gamma-1$ and $t$ (see \eqref{small-assumption1}, \eqref{small-assumption2}, \eqref{small-assumption3} and \eqref{small-assumption4}), and $\hat{\epsilon}>0$ is a small constant depending only on $\mu, \lambda$ and $\Omega$ (see \eqref{small-A-assumption1}).
\end{pro}

\begin{proof}
  Proposition \ref{prop-a-priori-estimate} can be directly derived from Lemma \ref{lem-essential-energy}-\ref{lem-rho-bound} below.
\end{proof}

We begin with the following standard energy estimate for $(\rho,u)$.

\begin{lem}\label{lem-essential-energy}
  Let $(\rho,u)$ be a smooth solution of \eqref{Large-CNS-eq}-\eqref{Navier-slip-condition} on $\Omega\times(0,T]$. Then it holds that
  \begin{equation}\label{essential-energy}
    \sup_{t\in[0,T]}\int(\frac{1}{2}\rho|u|^2+\frac{1}{\gamma-1}P(\rho))+\int_0^T\int(\mu|\curl u|^2+(2\mu+\lambda)|\div u|^2)\leq E_0.
  \end{equation}
\end{lem}

\begin{proof}
  Rewriting $\eqref{Large-CNS-eq}_1$ as
  \begin{equation}\label{Pressure-eq}
    P_t+\div(uP)+(\gamma-1)P\div u=0,
  \end{equation}
  and integrating over $\Omega$, then adding it to the $L^2$-inner product of $\eqref{Large-CNS-eq}_2$ with $u$ yields that
  \begin{equation}\label{essential-energy-estimate}
    \frac{d}{dt}\int(\frac{1}{2}\rho|u|^2+\frac{1}{\gamma-1}P(\rho))+\int(\mu|\curl u|^2+(2\mu+\lambda)|\div u|^2)+\mu\int_{\pa\Omega}Au\cdot u=0,
  \end{equation}
  where we have used the fact that
  \begin{equation*}
    \Delta u=\nabla\div u-\nabla\times\curl u.
  \end{equation*}
  A direct integration on $[0,T]$ gives the inequality \eqref{essential-energy}.  
\end{proof}

The following a priori estimate is essential to close the a priori assumption \eqref{a-priori-assumption}.
\begin{lem}\label{lem-essential-estimate}
  Under the conditions of Proposition \ref{prop-a-priori-estimate}, it holds that
  \begin{equation}\label{essential-estimate}
    \sup_{t\in[0,\sigma(T)]}\int\rho|u|^2+\int_0^{\sigma(T)}\int|\nabla u|^2\leq C(\tilde{\rho})\mathcal{E}_0.
  \end{equation}
\end{lem}

\begin{proof}
  Taking $L^2$-inner product of $\eqref{Large-CNS-eq}_2$ with $u$, it follows from integration by parts that
  \begin{equation}\label{essential-estimate1}
    \frac{d}{dt}\int\frac{1}{2}\rho|u|^2+\int(\mu|\curl u|^2+(2\mu+\lambda)|\div u|^2)+\mu\int_{\pa\Omega}Au\cdot u=\int P\div u.
  \end{equation}
  Integrating \eqref{essential-estimate1} over $[0,\sigma(T)]$, and using Cauchy's inequality and \eqref{essential-energy}, we have
  \begin{equation}\label{essential-estimate2}
    \begin{aligned}
    &\quad\sup_{t\in[0,\sigma(T)]}\int\frac{1}{2}\rho|u|^2+\int_0^{\sigma(T)}\int(\mu|\curl u|^2+\frac{1}{2}(2\mu+\lambda)|\div u|^2)\\
    &\leq\int\frac{1}{2}\rho_0|u_0|^2+C\int_0^{\sigma(T)}\int|P|^2\leq \int\frac{1}{2}\rho_0|u_0|^2+C(\tilde{\rho})\int P\\
    &\leq\frac{1}{2}\rho_0|u_0|^2+C(\tilde{\rho})(\gamma-1)E_0\leq C(\tilde{\rho})\mathcal{E}_0,
    \end{aligned}
  \end{equation}
  which together with \eqref{Hodge-decomposition1} implies \eqref{essential-estimate}.
\end{proof}

The next lemma gives the estimates on $A_1(T)$ and $A_2(T)$.
\begin{lem}\label{lem-A1-A2-control}
  Under the conditions of Proposition \ref{prop-a-priori-estimate}, it holds that
  \begin{equation}\label{A1-control}
    A_1(T)\leq C(\tilde{\rho})\mathcal{E}_0+C\mathcal{E}_0E_0+C\int_0^T\int\sigma P|\nabla u|^2+C\int_0^T\sigma\norm{\nabla u}_{L^3}^3+C(\tilde{\rho})\int_0^T\sigma\norm{\nabla u}_{L^2}^4,
  \end{equation}
  \begin{equation}\label{A2-control}
    \begin{aligned}
     A_2(T)&\leq C(\tilde{\rho})\mathcal{E}_0+CA_1^{\frac{3}{2}}(T)(1+C(\tilde{\rho})A_1(T))^{\frac{1}{2}}+CA_1(\sigma(T))+C(\tilde{\rho})A_1^3(T)(E_0+1)\\
    &\quad+C\int_0^T\sigma^3(\norm{\nabla u}_{L^4}^4+\norm{P|\nabla u|}_{L^2}^2+\norm{P-\bar{P}}_{L^4}^4),
    \end{aligned}
  \end{equation}
  provided $\mathcal{E}_0\leq\epsilon_1=1$.
\end{lem}

\begin{proof}
   For any integer $m\geq0$, multiplying $\eqref{Large-CNS-eq}_2$ by $\sigma^m\dot{u}$ and integrating over $\Omega$, we obtain
  \begin{equation}\label{A-control1}
    \begin{aligned}
    \int\sigma^m\rho|\dot{u}|^2&=-\int\sigma^m\dot{u}\cdot\nabla P+(2\mu+\lambda)\int\sigma^m\nabla\div u\cdot\dot{u}-\mu\int\sigma^m\nabla\times\curl u\cdot\dot{u}\\
    &=I_1+I_2+I_3.
    \end{aligned}
  \end{equation}
  We will estimate $I_1,I_2$ and $I_3$. First, a direct calculation by integration by parts yields that
  \begin{equation}\label{I_1-estimate}
    \begin{aligned}
    I_1&=-\int\sigma^m\dot{u}\cdot\nabla P\\
    &=\int\sigma^mP\div u_t-\int\sigma^mu\cdot\nabla u\cdot\nabla P\\
    &=\left(\int\sigma^mP\div u\right)_t-m\sigma^{m-1}\sigma'\int P\div u+\int\sigma^mP\nabla u:\nabla u^T+(\gamma-1)\int\sigma^mP(\div u)^2\\
    &\quad-\int_{\pa\Omega}\sigma^mPu\cdot\nabla u\cdot n,
    \end{aligned}
  \end{equation}
  where we have used the equation \eqref{Pressure-eq}
  \begin{equation*}
    P_t+\div(Pu)+(\gamma-1)P\div u=0.
  \end{equation*}
  
  Similarly, we estimate $I_2$ as
  \begin{equation}\label{I_2-estimate}
    \begin{aligned}
    I_2&=(2\mu+\lambda)\int\sigma^m\nabla\div u\cdot\dot{u}\\
    &=(2\mu+\lambda)\int_{\pa\Omega}\sigma^m\div u\dot{u}\cdot n-(2\mu+\lambda)\int\sigma^m\div u\div\dot{u}\\
    &=(2\mu+\lambda)\int_{\pa\Omega}\sigma^m\div uu\cdot\nabla u\cdot n-\frac{2\mu+\lambda}{2}\left(\int\sigma^m|\div u|^2\right)_t\\
    &\quad-(2\mu+\lambda)\int\sigma^m\div u\div(u\cdot\nabla u)+\frac{1}{2}m(\mu+\lambda)\sigma^{m-1}\sigma'\int|\div u|^2\\
    &=(2\mu+\lambda)\int_{\pa\Omega}\sigma^m\div uu\cdot\nabla u\cdot n-\frac{2\mu+\lambda}{2}\left(\int\sigma^m|\div u|^2\right)_t\\
    &\quad+\frac{2\mu+\lambda}{2}\int\sigma^m(\div u)^3-(2\mu+\lambda)\int\sigma^m\div u\nabla u:\nabla u^T+\frac{1}{2}m(2\mu+\lambda)\sigma^{m-1}\sigma'\int|\div u|^2.
    \end{aligned}
  \end{equation}
  Combining the boundary terms in \eqref{I_1-estimate} and \eqref{I_2-estimate}, we have
  \begin{equation}\label{boundary-term-A1-1}
    \begin{aligned}
    &\quad\int_{\pa\Omega}\sigma^m[(2\mu+\lambda)\div u-P]u\cdot\nabla u\cdot n\\
    &=\int_{\pa\Omega}\sigma^m(G-\bar{P})u\cdot\nabla u\cdot n=-\int_{\pa\Omega}\sigma^m(G-\bar{P})u\cdot\nabla n\cdot u\\
    &=-\int_{\pa\Omega}\sigma^m(u^{\bot}\times n)\cdot\nabla n_iu_iG+\bar{P}\int_{\pa\Omega}\sigma^mu\cdot\nabla n\cdot u\\
    &=-\int_{\pa\Omega}\sigma^m(\nabla n_i\times u^{\bot})\cdot nu_iG+\bar{P}\int_{\pa\Omega}\sigma^mu\cdot\nabla n\cdot u\\
    &=-\int\sigma^m\div((\nabla n_i\times u^{\bot})u_iG)+\bar{P}\int_{\pa\Omega}\sigma^mu\cdot\nabla n\cdot u\\
    &\leq C\int\sigma^m(|G||u|^2+|G||u||\nabla u|+|\nabla G||u|^2)+C\bar{P}\sigma^m\norm{u}_{H^1}^2\\
    &\leq C\sigma^m(\norm{G}_{L^2}\norm{u}_{L^4}^2+\norm{G}_{L^3}\norm{u}_{L^6}\norm{\nabla u}_{L^2}+\norm{\nabla G}_{L^2}\norm{u}_{L^4}^2)+C\bar{P}\sigma^m\norm{u}_{H^1}^2\\
    &\leq C\sigma^m\norm{G}_{H^1}\norm{u}_{H^1}^2+C\bar{P}\sigma^m\norm{u}_{H^1}^2\\
    &\leq C\sigma^m(\norm{\rho\dot{u}}_{L^2}\norm{\nabla u}_{L^2}^2+\norm{\nabla u}_{L^2}^3)+C\bar{P}\sigma^m\norm{\nabla u}_{L^2}^2\\
    &\leq\frac{1}{4}\sigma^m\norm{\sqrt{\rho}\dot{u}}_{L^2}^2+C(\tilde{\rho})\sigma^m\norm{\nabla u}_{L^2}^4+C\sigma^m(\norm{\nabla u}_{L^2}^3+\bar{P}\norm{\nabla u}_{L^2}^2),
    \end{aligned}
  \end{equation}
  where we have used Poincar\'{e} inequality, Sobolev imbedding, the simple argument on deriving \eqref{dot-u-transition1} and Lemma \ref{lem-curl-effective-viscous}.
  
  Finally, by \eqref{Navier-slip-condition} a similar computation on $I_3$ gives that
  \begin{equation}\label{I_3-estimate}
    \begin{aligned}
    I_3&=-\mu\int\sigma^m\nabla\times\curl u\cdot\dot{u}\\
    &=-\mu\int\sigma^m\curl u\cdot\curl\dot{u}-\mu\int_{\pa\Omega}\sigma^mn\times\curl u\dot{u}\\
    &=-\frac{\mu}{2}\left(\int\sigma^m|\curl u|^2+\int_{\pa\Omega}\sigma^mAu\cdot u\right)_t+\frac{\mu}{2}m\sigma^{m-1}\sigma'\left(\int|\curl u|^2+\int_{\pa\Omega}Au\cdot u\right)\\
    &\quad-\mu\int\sigma^m\curl u\cdot\curl(u\cdot\nabla u)-\mu\int_{\pa\Omega}Au\cdot(u\cdot\nabla u)\\
    &=-\frac{\mu}{2}\left(\int\sigma^m|\curl u|^2+\int_{\pa\Omega}\sigma^mAu\cdot u\right)_t+\frac{\mu}{2}m\sigma^{m-1}\sigma'\left(\int|\curl u|^2+\int_{\pa\Omega}Au\cdot u\right)\\
    &\quad+\frac{\mu}{2}\int\sigma^m|\curl u|^2\div u-\mu\int\sigma^m\curl u\cdot(\nabla u_i\times\nabla_i u)-\mu\int_{\pa\Omega}((Au)^{\bot}\times n)\cdot(u\cdot\nabla u)\\
    &=-\frac{\mu}{2}\left(\int\sigma^m|\curl u|^2+\int_{\pa\Omega}\sigma^mAu\cdot u\right)_t+\frac{\mu}{2}m\sigma^{m-1}\sigma'\left(\int|\curl u|^2+\int_{\pa\Omega}Au\cdot u\right)\\
    &\quad+\frac{\mu}{2}\int\sigma^m|\curl u|^2\div u-\mu\int\sigma^m\curl u\cdot(\nabla u_i\times\nabla_i u)-\mu\int_{\pa\Omega}((u\cdot\nabla u)\times (Au)^{\bot})\cdot n\\
    &\leq-\frac{\mu}{2}\left(\int\sigma^m|\curl u|^2+\int_{\pa\Omega}\sigma^mAu\cdot u\right)_t+\frac{\mu}{2}m\sigma^{m-1}\sigma'\left(\int|\curl u|^2+\int_{\pa\Omega}Au\cdot u\right)\\
    &\quad +C\sigma^m\norm{\nabla u}_{L^3}^3-\mu\int\div((u\cdot\nabla u)\times (Au)^{\bot})\\
    &\leq-\frac{\mu}{2}\left(\int\sigma^m|\curl u|^2+\int_{\pa\Omega}\sigma^mAu\cdot u\right)_t+Cm\sigma^{m-1}\sigma'\norm{\nabla u}_{L^2}^2+\frac{1}{4}\sigma^m\norm{\sqrt{\rho}\dot{u}}_{L^2}^2\\
    &\quad+C\sigma^m\norm{\nabla u}_{L^3}^3+C(\tilde{\rho})\sigma^m\norm{\nabla u}_{L^2}^4,
    \end{aligned}
  \end{equation}
  where we have used the simple fact that by \eqref{curl-effective-viscous-estimate2}, Sobolev imbedding and Poincar\'{e} inequality,
  \begin{equation*}
    \begin{aligned}
    &\quad\int\div((u\cdot\nabla u)\times (Au)^{\bot})\\
    &=\int Au^{\bot}\cdot(\nabla\times(u\cdot\nabla u))-\int(u\cdot\nabla u)\cdot(\nabla\times (Au)^{\bot})\\
    &=\int (Au)^{\bot}\cdot(u\cdot\nabla\curl u+\nabla u_i\times\nabla_i u)-\int(u\cdot\nabla u)\cdot(\nabla\times (Au)^{\bot})\\
    &\leq C\int(|u|^2|\nabla\curl u|+|u||\nabla u|^2+|u|^2|\nabla u|+|u|^3)\\
    &\leq C(\norm{\nabla\curl u}_{L^2}\norm{u}_{L^4}^2+\norm{u}_{L^3}\norm{\nabla u}_{L^3}^2+\norm{\nabla u}_{L^2}\norm{u}_{L^4}^2+\norm{u}_{L^3}^3)\\
    &\leq C(\norm{\rho\dot{u}}_{L^2}\norm{\nabla u}_{L^2}^2+\norm{\nabla u}_{L^2}^3+\norm{\nabla u}_{L^2}\norm{\nabla u}_{L^3}^2). 
    \end{aligned}
  \end{equation*}
  It follows from \eqref{A-control1}-\eqref{I_3-estimate} that
  \begin{equation}\label{A-control2}
    \begin{aligned}
    &\quad\left(\int\sigma^m(\mu|\curl u|^2+(2\mu+\lambda)|\div u|^2-2P\div u)+\mu\int_{\pa\Omega}\sigma^mAu\cdot u\right)_t+\int\sigma^m\rho|\dot{u}|^2\\
    &\leq Cm\sigma^{m-1}\sigma'(\int P|\nabla u|+\norm{\nabla u}_{L^2}^2)+C\sigma^m(\int P|\nabla u|^2+\norm{\nabla u}_{L^3}^3+\bar{P}\norm{\nabla u}_{L^2}^2)\\
    &\quad+C(\tilde{\rho})\sigma^m\norm{\nabla u}_{L^2}^4.
    \end{aligned}
  \end{equation}
  Then integrating \eqref{A-control2} over $[0,T]$, and using \eqref{essential-energy}, \eqref{a-priori-assumption}, \eqref{curl-effective-viscous-estimate1} and \eqref{essential-estimate}, we have that for any integer $m\geq1$,
  \begin{equation}\label{A-control3}
    \begin{aligned}
    &\quad\sup_{t\in[0,T]}\sigma^m\norm{\nabla u}_{L^2}^2+\int_0^T\int\sigma^m\rho|\dot{u}|^2\\
    &\leq C(\tilde{\rho})\mathcal{E}_0+C\int_0^{\sigma(T)}\int(P^2+|\nabla u|^2)+C\int_0^T\int\sigma^mP|\nabla u|^2+C\int_0^T\bar{P}\norm{\nabla u}_{L^2}^2\\
    &\quad+C\int_0^T\sigma^m\norm{\nabla u}_{L^3}^3+C(\tilde{\rho})\int_0^T\sigma^m\norm{\nabla u}_{L^2}^4\\
    &\leq C(\tilde{\rho})\mathcal{E}_0+C\mathcal{E}_0E_0+C\int_0^T\int\sigma^mP|\nabla u|^2+C\int_0^T\sigma^m\norm{\nabla u}_{L^3}^3+C(\tilde{\rho})\int_0^T\sigma^m\norm{\nabla u}_{L^2}^4, 
    \end{aligned}
  \end{equation}
  which implies \eqref{A1-control} by taking $m=1$.
  
  Now, we turn to prove \eqref{A2-control}. Recalling \eqref{momentum-equation} as
  \begin{equation}\label{momentum-equation2}
    \rho\dot{u}=\nabla G-\mu\nabla\times\curl u,
  \end{equation}
  then taking $\sigma^m\dot{u}_j[\pa_t+\div(u\cdot)]$ on the $j$-th component of $\eqref{momentum-equation2}$, summing over $j$, and integrating over $\Omega$ yields
  \begin{equation}\label{A-control-1}
    \begin{aligned}
    &\quad\left(\frac{1}{2}\int\sigma^m\rho|\dot{u}|^2\right)_t-\frac{1}{2}m\sigma^{m-1}\sigma'\int\rho|\dot{u}|^2\\
    &=\int\sigma^m(\dot{u}\cdot\nabla G_t+\dot{u}_j\div(u\pa_jG))+\mu\int\sigma^m(-\dot{u}\cdot\nabla\times\curl u_t-\dot{u}_j\div(u(\nabla\times\curl u)_j))\\
    &=J_1+\mu J_2.
    \end{aligned}
  \end{equation}
  For $J_1$, by virtue of \eqref{Navier-slip-condition} and \eqref{Pressure-eq}, we have
  \begin{equation}\label{J_1-estimate1}
    \begin{aligned}
    J_1&=\int\sigma^m\dot{u}\cdot\nabla G_t+\int\sigma^m\dot{u}_j\div(u\pa_jG)\\
    &=\int_{\pa\Omega}\sigma^mG_t\dot{u}\cdot n-\int\sigma^mF_t\div\dot{u}-\int\sigma^mu\cdot\nabla\dot{u}\cdot\nabla G\\
    &=\int_{\pa\Omega}\sigma^mG_t\dot{u}\cdot n-(2\mu+\lambda)\int\sigma^m|\div\dot{u}|^2+(2\mu+\lambda)\int\sigma^m\div\dot{u}\nabla u:\nabla u^T\\
    &\quad+\int\sigma^m\div\dot{u}u\cdot\nabla G-\gamma\int\sigma^mP\div u\div\dot{u}+(\gamma-1)\overline{P\div u}\int\sigma^m\div\dot{u}-\int\sigma^mu\cdot\nabla\dot{u}\cdot\nabla G\\
    &\leq\int_{\pa\Omega}\sigma^mG_t\dot{u}\cdot n-(2\mu+\lambda)\int\sigma^m|\div\dot{u}|^2+\delta\sigma^m\norm{\nabla\dot{u}}_{L^2}^2\\
    &\quad+C(\delta)\sigma^m(\norm{\nabla u}_{L^2}^2\norm{\nabla G}_{L^3}^2+\norm{\nabla u}_{L^4}^4+\norm{P|\nabla u|}_{L^2}^2),
    \end{aligned}
  \end{equation}
  where we have used
  \begin{equation*}
    \begin{aligned}
    F_t&=(2\mu+\lambda)\div u_t-(P_t-\bar{P}_t)\\
    &=(2\mu+\lambda)\div\dot{u}-(2\mu+\lambda)\div(u\cdot\nabla u)+u\cdot\nabla P+\gamma P\div u-(\gamma-1)\overline{P\div u}\\
    &=(2\mu+\lambda)\div\dot{u}-(2\mu+\lambda)\nabla u:\nabla u^T-u\cdot\nabla G+\gamma P\div u-(\gamma-1)\overline{P\div u}.
    \end{aligned}
  \end{equation*}
  For the boundary term in \eqref{A-control-1}, we have
  \begin{equation}\label{boundary-term-A2-1}
    \begin{aligned}
    &\quad\int_{\pa\Omega}\sigma^mG_t\dot{u}\cdot n
    =-\int_{\pa\Omega}\sigma^mG_tu\cdot\nabla n\cdot u\\
    &=-\left(\int_{\pa\Omega}\sigma^mG(u\cdot\nabla n\cdot u)\right)_t+m\sigma^{m-1}\sigma'\int_{\pa\Omega}G(u\cdot\nabla n\cdot u)\\
    &\quad+\sigma^m\int_{\pa\Omega}G(\dot{u}\cdot\nabla n\cdot u+u\cdot\nabla n\cdot\dot{u})-\sigma^m\int_{\pa\Omega}G((u\cdot\nabla u)\cdot\nabla n\cdot u+u\cdot\nabla n\cdot(u\cdot\nabla u))\\
    &\leq-\left(\int_{\pa\Omega}\sigma^mG(u\cdot\nabla n\cdot u)\right)_t+Cm\sigma^{m-1}\sigma'\norm{\nabla u}_{L^2}^2\norm{G}_{H^1}+C\sigma^m\norm{G}_{H^1}\norm{\dot{u}}_{H^1}\norm{\nabla u}_{L^2}\\
    &\quad-\sigma^m\int_{\pa\Omega}G((u\cdot\nabla u)\cdot\nabla n\cdot u+u\cdot\nabla n\cdot(u\cdot\nabla u)),
    \end{aligned}
  \end{equation}
  where we have used the similar technical argument in \eqref{boundary-term-A1-1} to derive that
  \begin{equation}\label{boundary-term-A2-transition1}
  \begin{aligned}
    &\quad\int_{\pa\Omega}G(\dot{u}\cdot\nabla n\cdot u+u\cdot\nabla n\cdot\dot{u})=\int_{\pa\Omega}Gu\cdot(\nabla n+\nabla n^T)\cdot\dot{u}\\
    &=\int_{\pa\Omega}(u^{\bot}\times n)\cdot(\nabla n+\nabla n^T)\cdot\dot{u}G=\int_{\pa\Omega}((\nabla n+\nabla n^T)\cdot\dot{u}G\times u^{\bot})\cdot n\\
    &=\int\div((\nabla n+\nabla n^T)\cdot\dot{u}G\times u^{\bot})\\
    &\leq C\int(|\nabla G||\dot{u}||u|+|G|(|\nabla\dot{u}||u|+|\dot{u}||u|+|\dot{u}||\nabla u|)\\
    &\leq C\norm{G}_{H^1}\norm{\dot{u}}_{H^1}\norm{u}_{H^1}\leq C\norm{G}_{H^1}\norm{\dot{u}}_{H^1}\norm{\nabla u}_{L^2}, 
  \end{aligned}
  \end{equation}
  and
  \begin{equation}\label{boundary-term-A2-transition2}
    \left|\int_{\pa\Omega}G(u\cdot\nabla n\cdot u)\right|\leq C\norm{G}_{H^1}\norm{u}_{H^1}^2\leq C\norm{G}_{H^1}\norm{\nabla u}_{L^2}^2.
  \end{equation}
  Similar to the proof of \eqref{boundary-term-A1-1}, we have from Lemma \ref{lem-curl-effective-viscous}, Sobolev imbedding, Poincar\'{e} inequality \eqref{Poincare-inequality} and Lemma \ref{lem-dot-u} that
  \begin{equation}\label{boundary-term-A2-2}
    \begin{aligned}
    &\quad-\int_{\pa\Omega}G(u\cdot\nabla u)\cdot\nabla n\cdot u=-\int_{\pa\Omega}u^{\bot}\times n\cdot\nabla u_i\nabla_in\cdot uG\\
    &=\int_{\pa\Omega}(u^{\bot}\times\nabla u_i\nabla_in\cdot uG)\cdot n=\int\div(u^{\bot}\times\nabla u_i\nabla_in\cdot uG)\\
    &=\int(u^{\bot}\times\nabla u_i)\cdot\nabla(\nabla_in\cdot uG)+\int(\nabla\times u^{\bot})\cdot\nabla u_i\nabla_in\cdot uG\\
    &\leq C\int|\nabla G||u|^2|\nabla u|+C\int|G|(|\nabla u|^2|u|+|\nabla u||u|^2)\\
    &\leq C\norm{\nabla G}_{L^4}\norm{\nabla u}_{L^4}\norm{u}_{L^4}^2+C\norm{G}_{L^3}\norm{\nabla u}_{L^4}^2\norm{u}_{L^6}+C\norm{G}_{L^6}\norm{\nabla u}_{L^2}\norm{u}_{L^6}^2\\
    &\leq C\norm{\nabla G}_{L^4}\norm{\nabla u}_{L^4}\norm{\nabla u}_{L^2}^2+C\norm{\nabla G}_{L^2}\norm{\nabla u}_{L^4}^2\norm{\nabla u}_{L^2}\\
    &\leq2\delta\norm{\nabla\dot{u}}_{L^2}^2+C(\norm{\nabla u}_{L^4}^4+\norm{P-\bar{P}}_{L^4}^4)+C(\delta,\tilde{\rho})(\norm{\nabla u}_{L^2}^8+\norm{\sqrt{\rho}\dot{u}}_{L^2}^2\norm{\nabla u}_{L^2}^4),
    \end{aligned}
  \end{equation}
  where we have used the simple fact that
  \begin{equation*}
    \div(a\times b)=(\nabla\times a)\cdot b-(\nabla\times b)\cdot a,
  \end{equation*}
  and the following estimates as
  \begin{equation}\label{boundary-term-A2-2-transition1}
    \begin{aligned}
    &\quad\norm{\nabla G}_{L^4}\norm{\nabla u}_{L^4}\norm{\nabla u}_{L^2}^2\\
    &\leq C(\norm{\rho\dot{u}}_{L^4}+\norm{\nabla u}_{L^2}+\norm{P-\bar{P}}_{L^4})\norm{\nabla u}_{L^4}\norm{\nabla u}_{L^2}^2\\
    &\leq C(\tilde{\rho})(\norm{\nabla\dot{u}}_{L^2}+\norm{\nabla u}_{L^2}^2)\norm{\nabla u}_{L^4}\norm{\nabla u}_{L^2}^2+C(\norm{\nabla u}_{L^2}+\norm{P-\bar{P}}_{L^4})\norm{\nabla u}_{L^4}\norm{\nabla u}_{L^2}^2\\
    &\leq\delta\norm{\nabla\dot{u}}_{L^2}^2+C\norm{\nabla u}_{L^4}^4+C(\delta,\tilde{\rho})\norm{\nabla u}_{L^2}^8+C\norm{P-\bar{P}}_{L^4}^4
    \end{aligned}
  \end{equation}
  and
  \begin{equation}\label{boundary-term-A2-2-transition2}
    \begin{aligned}
    &\quad\norm{\nabla G}_{L^2}\norm{\nabla u}_{L^4}^2\norm{\nabla u}_{L^2}\\
    &\leq C(\norm{\rho\dot{u}}_{L^2}+\norm{\nabla u}_{L^2})\norm{\nabla u}_{L^4}^2\norm{\nabla u}_{L^2}\\
    &\leq C\norm{\nabla u}_{L^4}^4+C\norm{\rho\dot{u}}_{L^2}^2\norm{\nabla u}_{L^2}^2\\
    &\leq C\norm{\nabla u}_{L^4}^4+C(\tilde{\rho})\norm{\rho\dot{u}}_{L^2}\norm{\dot{u}}_{L^6}\norm{\nabla u}_{L^2}^2\\
    &\leq C\norm{\nabla u}_{L^4}^4+C(\tilde{\rho})\norm{\sqrt{\rho}\dot{u}}_{L^2}(\norm{\nabla\dot{u}}_{L^2}+\norm{\nabla u}_{L^2}^2)\norm{\nabla u}_{L^2}^2\\
    &\leq\delta\norm{\nabla\dot{u}}_{L^2}^2+C\norm{\nabla u}_{L^4}^4+C(\delta,\tilde{\rho})\norm{\sqrt{\rho}\dot{u}}_{L^2}^2\norm{\nabla u}_{L^2}^4.
    \end{aligned}
  \end{equation}
  Similarly, we have
  \begin{equation}\label{boundary-term-A2-3}
    \begin{aligned}
     &\quad-\int_{\pa\Omega}Gu\cdot\nabla n\cdot(u\cdot\nabla u)\\
     &\leq 2\delta\norm{\nabla\dot{u}}_{L^2}^2+C(\norm{\nabla u}_{L^4}^4+\norm{P-\bar{P}}_{L^4}^4)+C(\delta,\tilde{\rho})(\norm{\nabla u}_{L^2}^8+\norm{\sqrt{\rho}\dot{u}}_{L^2}^2\norm{\nabla u}_{L^2}^4).
    \end{aligned}
  \end{equation}
  Indeed, from Lemma \ref{lem-curl-effective-viscous} and Lemma \ref{lem-dot-u}, we get
  \begin{equation}\label{A-control-transition1}
    \norm{G}_{H^1}+\norm{\curl u}_{H^1}\leq C(\norm{\rho\dot{u}}_{L^2}+\norm{\nabla u}_{L^2}),
  \end{equation}
  and for any $p\in[2,6]$,
  \begin{equation}\label{A-control-transition2}
    \begin{aligned}
    \norm{\nabla G}_{L^p}+\norm{\nabla\curl u}_{L^p}&\leq C(\norm{\rho\dot{u}}_{L^p}+\norm{\nabla u}_{L^2}+\norm{P-\bar{P}}_{L^p})\\
    &\leq C(\tilde{\rho})\norm{\dot{u}}_{L^6}+C(\norm{\nabla u}_{L^2}+\norm{P-\bar{P}}_{L^p})\\
    &\leq C(\tilde{\rho})(\norm{\nabla\dot{u}}_{L^2}+\norm{\nabla u}_{L^2}^2)+C(\norm{\nabla u}_{L^2}+\norm{P-\bar{P}}_{L^p}).
    \end{aligned}
  \end{equation}
  Consequently, it holds that
  \begin{equation}\label{A2-estimate-transition1}
    \begin{aligned}
    &\quad\norm{\nabla u}_{L^2}^2\norm{\nabla G}_{L^3}^2\\
    &\leq C\norm{\nabla u}_{L^2}^2(\norm{\rho\dot{u}}_{L^3}^2+\norm{\nabla u}_{L^2}^2+\norm{P-\bar{P}}_{L^3}^2)\\ 
    &\leq C(\tilde{\rho})\norm{\nabla u}_{L^2}^2\norm{\sqrt{\rho}\dot{u}}_{L^2}\norm{\dot{u}}_{L^6}+C(\norm{\nabla u}_{L^2}^4+\norm{P-\bar{P}}_{L^3}^4)\\
    &\leq C(\tilde{\rho})\norm{\nabla u}_{L^2}^2\norm{\sqrt{\rho}\dot{u}}_{L^2}(\norm{\nabla\dot{u}}_{L^2}+\norm{\nabla u}_{L^2}^2)+C(\norm{\nabla u}_{L^2}^4+\norm{P-\bar{P}}_{L^3}^4)\\
    &\leq\delta\norm{\nabla\dot{u}}_{L^2}^2+C(\delta,\tilde{\rho})\norm{\sqrt{\rho}\dot{u}}_{L^2}^2\norm{\nabla u}_{L^2}^4+C(\norm{\nabla u}_{L^2}^4+\norm{P-\bar{P}}_{L^3}^4)
    \end{aligned}
  \end{equation}
  and
  \begin{equation}\label{A2-estimate-transition2}
    \begin{aligned}
    &\quad\norm{G}_{H^1}\norm{\dot{u}}_{H^1}\norm{\nabla u}_{L^2}\\
    &\leq C(\norm{\rho\dot{u}}_{L^2}+\norm{\nabla u}_{L^2})(\norm{\nabla\dot{u}}_{L^2}+\norm{\nabla u}_{L^2}^2)\norm{\nabla u}_{L^2}\\
    &\leq\delta\norm{\nabla\dot{u}}_{L^2}^2+C(\delta)\norm{\rho\dot{u}}_{L^2}^2\norm{\nabla u}_{L^2}^2+C(\delta)\norm{\nabla u}_{L^2}^4\\
    &\leq\delta\norm{\nabla\dot{u}}_{L^2}^2+C(\delta,\tilde{\rho})\norm{\rho\dot{u}}_{L^2}\norm{\dot{u}}_{L^6}\norm{\nabla u}_{L^2}^2+C(\delta)\norm{\nabla u}_{L^2}^4\\
    &\leq2\delta\norm{\nabla\dot{u}}_{L^2}^2+C(\delta,\tilde{\rho})\norm{\sqrt{\rho}\dot{u}}_{L^2}^2\norm{\nabla u}_{L^2}^4+C(\delta)\norm{\nabla u}_{L^2}^4. 
    \end{aligned}
  \end{equation}
  Then combining \eqref{J_1-estimate1}, \eqref{boundary-term-A2-1}, \eqref{boundary-term-A2-2}, \eqref{boundary-term-A2-3}, \eqref{A2-estimate-transition1} and \eqref{A2-estimate-transition2}, we obtain
  \begin{equation}\label{J_1-estimate2}
    \begin{aligned}
    J_1&\leq-\left(\int_{\pa\Omega}\sigma^mG(u\cdot\nabla n\cdot u)\right)_t+Cm\sigma^{m-1}\sigma'(\norm{\nabla u}_{L^2}^2(\norm{\sqrt{\rho}\dot{u}}_{L^2}^2+C(\tilde{\rho}))+\norm{\nabla u}_{L^2}^4)\\
    &\quad-(2\mu+\lambda)\int\sigma^m|\div\dot{u}|^2+C\delta\sigma^m\norm{\nabla\dot{u}}_{L^2}^2
    +C\sigma^m(\norm{\nabla u}_{L^2}^4+\norm{P-\bar{P}}_{L^4}^4)\\
    &\quad+C(\delta,\tilde{\rho})\sigma^m(\norm{\sqrt{\rho}\dot{u}}_{L^2}^2\norm{\nabla u}_{L^2}^4+\norm{\nabla u}_{L^2}^8)+C(\delta)\sigma^m(\norm{\nabla u}_{L^4}^4+\norm{P|\nabla u|}_{L^2}^2). 
    \end{aligned}
  \end{equation}
  Similarly, for $J_2$, we have
  \begin{equation}\label{J_2-estimate1}
    \begin{aligned}
    J_2&=\int\sigma^m(-\dot{u}\cdot\nabla\times\curl u_t-\dot{u}_j\div(u(\nabla\times\curl u)_j))\\
    &=-\int\sigma^m|\curl\dot{u}|^2+\int\sigma^m\curl\dot{u}\cdot(\nabla u_i\times\nabla_iu)+\int\sigma^mu\cdot\nabla\curl u\cdot\curl\dot{u}\\
    &\quad+\int_{\pa\Omega}\sigma^m\curl u_t\times n\cdot\dot{u}+\int\sigma^mu\cdot\nabla\dot{u}\cdot(\nabla\times\curl u)\\
    &\leq-\int\sigma^m|\curl\dot{u}|^2-\int_{\pa\Omega}\sigma^mA\dot{u}\cdot\dot{u}
    +\delta\sigma^m\norm{\nabla\dot{u}}_{L^2}^2+C(\delta)\sigma^m(\norm{\nabla u}_{L^4}^4+\norm{u|\nabla\curl u|}_{L^2}^2)\\
    &\quad+C\sigma^m(\norm{\dot{u}}_{L^6}\norm{\nabla u}_{L^2}\norm{u}_{L^3}+\norm{\nabla u}_{L^4}^2\norm{\dot{u}}_{L^2}+\norm{u}_{L^3}\norm{\nabla u}_{L^2}\norm{\dot{u}}_{L^6})\\
    &\leq-\int\sigma^m|\curl\dot{u}|^2-\int_{\pa\Omega}\sigma^mA\dot{u}\cdot\dot{u}
    +3\delta\sigma^m\norm{\nabla\dot{u}}_{L^2}^2+C(\delta)\sigma^m\norm{\nabla u}_{L^4}^4\\
    &\quad+C(\delta,\tilde{\rho})\sigma^m\norm{\sqrt{\rho}\dot{u}}_{L^2}^2\norm{\nabla u}_{L^2}^4,
    \end{aligned}
  \end{equation}
  where we have applied the facts that
  \begin{equation*}
    \curl u_t=\curl\dot{u}-u\cdot\nabla\curl u-\nabla u_i\times\nabla_iu
  \end{equation*}
  and
  \begin{equation}\label{boundary-term-A2-4}
    \begin{aligned}
    &\quad\int_{\pa\Omega}\curl u_t\times n\cdot\dot{u}=-\int_{\pa\Omega}Au_t\cdot\dot{u}\\
    &=-\int_{\pa\Omega}A\dot{u}\cdot\dot{u}+\int_{\pa\Omega}(u\cdot\nabla u)\cdot A\cdot\dot{u}\\
    &=-\int_{\pa\Omega}A\dot{u}\cdot\dot{u}+\int_{\pa\Omega}(u^{\bot}\times n\cdot\nabla u)\cdot A\dot{u}\\
    &=-\int_{\pa\Omega}A\dot{u}\cdot\dot{u}+\int_{\pa\Omega}(\nabla u_i(A\dot{u})_i\times u^{\bot})\cdot n\\
    &=-\int_{\pa\Omega}A\dot{u}\cdot\dot{u}+\int\div(\nabla u_i(A\dot{u})_i\times u^{\bot})\\
    &=-\int_{\pa\Omega}A\dot{u}\cdot\dot{u}+\int\nabla\times(\nabla u_i(A\dot{u})_i)\cdot u^{\bot}-\int\nabla\times u^{\bot}\cdot\nabla u_i(A\dot{u})_i\\
    &=-\int_{\pa\Omega}A\dot{u}\cdot\dot{u}+\int\nabla(A\dot{u})_i\times\nabla u_i\cdot u^{\bot}-\int\nabla\times u^{\bot}\cdot\nabla u_i(A\dot{u})_i,
    \end{aligned}
  \end{equation}
  and the simple estimate from Lemma \ref{lem-curl-effective-viscous} and Lemma \ref{lem-dot-u} 
  \begin{equation*}
    \begin{aligned}
    \norm{u|\nabla\curl u|}_{L^2}^2&\leq\norm{u}_{L^6}^2\norm{\nabla\curl u}_{L^3}^2\leq C\norm{\nabla u}_{L^2}^2(\norm{\rho\dot{u}}_{L^3}^2+\norm{\nabla u}_{L^2}^2+\norm{P-\bar{P}}_{L^3}^2)\\
    &\leq C\norm{\nabla u}_{L^2}^2(\norm{\rho\dot{u}}_{L^2}\norm{\rho\dot{u}}_{L^6}+\norm{\nabla u}_{L^2}^2+\norm{P-\bar{P}}_{L^3}^2)\\
    &\leq C(\tilde{\rho})\norm{\nabla u}_{L^2}^2\norm{\rho\dot{u}}_{L^2}(\norm{\nabla\dot{u}}_{L^2}+\norm{\nabla u}_{L^2}^2)+C(\norm{\nabla u}_{L^2}^4+\norm{P-\bar{P}}_{L^3}^4)\\
    &\leq\delta\norm{\nabla\dot{u}}_{L^2}^2+C(\delta,\tilde{\rho})\norm{\sqrt{\rho}\dot{u}}_{L^2}^2\norm{\nabla u}_{L^2}^4+C(\norm{\nabla u}_{L^2}^4+\norm{P-\bar{P}}_{L^3}^4). 
    \end{aligned}
  \end{equation*}
  Therefore, combining \eqref{A-control-1}, \eqref{J_1-estimate2} and \eqref{J_2-estimate1} gives that
  \begin{equation}\label{A-control-2}
    \begin{aligned}
    &\quad\left(\int\sigma^m\rho|\dot{u}|^2+2\int_{\pa\Omega}\sigma^mu\cdot\nabla n\cdot uG\right)_t+2\sigma^m\int(\mu|\curl\dot{u}|^2+(2\mu+\lambda)|\div\dot{u}|^2)\\
    &\leq Cm\sigma^{m-1}\sigma'(\norm{\nabla u}_{L^2}^2(\norm{\sqrt{\rho}\dot{u}}_{L^2}^2+C(\tilde{\rho}))+\norm{\sqrt{\rho}\dot{u}}_{L^2}^2+\norm{\nabla u}_{L^2}^4)+C\delta\sigma^m\norm{\nabla\dot{u}}_{L^2}^2
    \\
    &\quad+C(\delta)\sigma^m(\norm{\nabla u}_{L^4}^4+\norm{P|\nabla u|}_{L^2}^2)+C(\delta,\tilde{\rho})\sigma^m(\norm{\nabla u}_{L^2}^8+\norm{\sqrt{\rho}\dot{u}}_{L^2}^2\norm{\nabla u}_{L^2}^4)\\
    &\quad+C(\norm{\nabla u}_{L^2}^4+\norm{P-\bar{P}}_{L^4}^4). 
    \end{aligned}
  \end{equation}
  Thus, using Lemma \ref{lem-dot-u} and choosing $\delta>0$ sufficiently small yields that
  \begin{equation}\label{A-control-3}
    \begin{aligned}
    &\quad\left(\int\sigma^m\rho|\dot{u}|^2+2\int_{\pa\Omega}\sigma^mu\cdot\nabla n\cdot uG\right)_t+\sigma^m\int(\mu|\curl\dot{u}|^2+(2\mu+\lambda)|\div\dot{u}|^2)\\
    &\leq Cm\sigma^{m-1}\sigma'(\norm{\nabla u}_{L^2}^2(\norm{\sqrt{\rho}\dot{u}}_{L^2}^2+C(\tilde{\rho}))+\norm{\sqrt{\rho}\dot{u}}_{L^2}^2+\norm{\nabla u}_{L^2}^4)
    \\
    &\quad+C\sigma^m(\norm{\nabla u}_{L^4}^4+\norm{P|\nabla u|}_{L^2}^2+\norm{P-\bar{P}}_{L^4}^4)+C(\tilde{\rho})\sigma^m(\norm{\nabla u}_{L^2}^8+\norm{\sqrt{\rho}\dot{u}}_{L^2}^2\norm{\nabla u}_{L^2}^4).
    \end{aligned}
  \end{equation}
  Integrating the above inequality over $[0,T]$, taking $m=3$ and using \eqref{a-priori-assumption}, \eqref{essential-energy}, \eqref{essential-estimate}, \eqref{boundary-term-A2-transition2} and Lemma \ref{lem-dot-u}, we have
  \begin{equation}\label{A-control-4}
    \begin{aligned}
    &\quad\sup_{t\in[0,T]}\sigma^3\int\rho|\dot{u}|^2+\int_0^T\int\sigma^3|\nabla\dot{u}|^2\\
    &\leq C\sup_{t\in[0,T]}\sigma^3\norm{G}_{H^1}\norm{\nabla u}_{L^2}^2+C\int_0^{\sigma(T)}\sigma^2(\norm{\nabla u}_{L^2}^2(\norm{\sqrt{\rho}\dot{u}}_{L^2}^2+C(\tilde{\rho}))+\norm{\sqrt{\rho}\dot{u}}_{L^2}^2+\norm{\nabla u}_{L^2}^4)\\
    &\quad+C\int_0^T\sigma^3(\norm{\nabla u}_{L^4}^4+\norm{P|\nabla u|}_{L^2}^2+\norm{P-\bar{P}}_{L^4}^4)+C(\tilde{\rho})\int_0^T\sigma^3(\norm{\nabla u}_{L^2}^8+\norm{\sqrt{\rho}\dot{u}}_{L^2}^2\norm{\nabla u}_{L^2}^4)\\
    &\leq \sup_{t\in[0,T]}\sigma^3(\frac{1}{4}\norm{\sqrt{\rho}\dot{u}}_{L^2}^2+C\norm{\nabla u}_{L^2}^3+C(\tilde{\rho})\norm{\nabla u}_{L^2}^4)+C(\tilde{\rho})\mathcal{E}_0+CA_1^2(\sigma(T))+CA_1(\sigma(T))\\
    &\quad+C(\tilde{\rho})A_1(\sigma(T))\mathcal{E}_0+C(\tilde{\rho})A_1^3(T)(1+E_0)+C\int_0^T\sigma^3(\norm{\nabla u}_{L^4}^4+\norm{P|\nabla u|}_{L^2}^2+\norm{P-\bar{P}}_{L^4}^4),
    \end{aligned}
  \end{equation}
  which implies
  \begin{equation}\label{A-control-5}
    \begin{aligned}
    A_2(T)&\leq CA_1^{\frac{3}{2}}(T)+C(\tilde{\rho})A_1^2(T)+C(\tilde{\rho})\mathcal{E}_0+CA_1^2(\sigma(T))+CA_1(\sigma(T))
    +C(\tilde{\rho})A_1(\sigma(T))\mathcal{E}_0\\
    &\quad+C(\tilde{\rho})A_1^3(T)(1+E_0)+C\int_0^T\sigma^3(\norm{\nabla u}_{L^4}^4+\norm{P|\nabla u|}_{L^2}^2+\norm{P-\bar{P}}_{L^4}^4)\\
    &\leq C(\tilde{\rho})\mathcal{E}_0+CA_1^{\frac{3}{2}}(T)(1+C(\tilde{\rho})A_1(T))^{\frac{1}{2}}+CA_1(\sigma(T))+C(\tilde{\rho})A_1^3(T)(E_0+1)\\
    &\quad+C\int_0^T\sigma^3(\norm{\nabla u}_{L^4}^4+\norm{P|\nabla u|}_{L^2}^2+\norm{P-\bar{P}}_{L^4}^4)
    \end{aligned}
  \end{equation}
  provided $\mathcal{E}_0\leq1$. Thus, we complete the proof of Lemma \ref{lem-A1-A2-control}.    
\end{proof}

\begin{lem}\label{lem-small-time-estimate}
  Under the conditions of Proposition \ref{prop-a-priori-estimate}, it holds that
  \begin{equation}\label{small-time-estimate1}
    \sup_{t\in[0,\sigma(T)]}\int|\nabla u|^2+\int_0^{\sigma(T)}\int\rho|\dot{u}|^2\leq C(\tilde{\rho}, M),
  \end{equation}
  \begin{equation}\label{small-time-estimate2}
    \sup_{t\in[0,\sigma(T)]}t\int\rho|\dot{u}|^2+\int_0^{\sigma(T)}\int t|\nabla\dot{u}|^2\leq C(\tilde{\rho}, M),
  \end{equation}
  \begin{equation}\label{small-time-estimate3}
    \sup_{t\in[0,T]}\int|\nabla u|^2+\int_0^{T}\int\rho|\dot{u}|^2\leq C(\tilde{\rho}, M),
  \end{equation}
  \begin{equation}\label{small-time-estimate4}
    \sup_{t\in[0,T]}\sigma\int\rho|\dot{u}|^2+\int_0^{T}\int \sigma|\nabla\dot{u}|^2\leq C(\tilde{\rho}, M),
  \end{equation}
  provided 
\begin{equation}\label{small-assumption1}
  \mathcal{E}_0\leq\epsilon_2=\min\{\epsilon_1,(4C(\tilde{\rho}))^{-12})\}
\end{equation}
with $\epsilon_1$ defined in Lemma \ref{lem-A1-A2-control}.
\end{lem}

\begin{proof}
  Multiplying $\eqref{Large-CNS-eq}_2$ by $u_t$ and integrating over $\Omega$, we get
  \begin{equation}\label{small-T-estimate1}
  \begin{aligned}
    &\quad\frac{d}{dt}\left(\int(\frac{\mu}{2}|\curl u|^2+\frac{2\mu+\lambda}{2}|\div u|^2-P\div u)+\frac{\mu}{2}\int_{\pa\Omega}Au\cdot u\right)+\int\rho|\dot{u}|^2\\
    &=\int\rho\dot{u}\cdot(u\cdot\nabla u)-\int P_t\div u\\
    &=\int\rho\dot{u}\cdot(u\cdot\nabla u)-\int Pu\cdot\nabla\div u+(\gamma-1)\int P|\div u|^2\\
    &=\int\rho\dot{u}\cdot(u\cdot\nabla u)-\frac{1}{2\mu+\lambda}\int Pu\cdot\nabla G+\frac{1}{2(2\mu+\lambda)}\int\div uP^2+(\gamma-1)\int P|\div u|^2\\
    &\leq C(\tilde{\rho})\norm{\sqrt{\rho}\dot{u}}_{L^2}\norm{\rho^{\frac{1}{3}}u}_{L^3}\norm{\nabla u}_{L^6}+C\norm{P}_{L^3}\norm{u}_{L^6}\norm{\nabla G}_{L^2}+C\norm{\nabla u}_{L^2}\norm{P}_{L^4}^2\\
    &\quad+C(\tilde{\rho})(\gamma-1)\norm{\nabla u}_{L^2}^2\\
    &\leq C(\tilde{\rho})\norm{\sqrt{\rho}\dot{u}}_{L^2}\norm{\rho^{\frac{1}{3}}u}_{L^3}(\norm{\rho\dot{u}}_{L^2}+\norm{P-\bar{P}}_{L^6}+\norm{\nabla u}_{L^2})\\
    &\quad+C\norm{P}_{L^3}\norm{\nabla u}_{L^2}(\norm{\rho\dot{u}}_{L^2}+\norm{\nabla u}_{L^2})+C(\tilde{\rho})\norm{\nabla u}_{L^2}^2+C\norm{P}_{L^4}^4\\
    &\leq (C(\tilde{\rho})\norm{\rho^{\frac{1}{3}}u}_{L^3}+\frac{1}{4})\norm{\sqrt{\rho}\dot{u}}_{L^2}^2+C(\tilde{\rho})\norm{\rho^{\frac{1}{3}}u}_{L^3}(\norm{\nabla u}_{L^2}^2+\norm{P}_{L^6}^2)+C(\tilde{\rho})(\norm{\nabla u}_{L^2}^2+\norm{P}_{L^1})\\
    &\leq(C(\tilde{\rho})\norm{\rho^{\frac{1}{3}}u}_{L^3}+\frac{1}{4})\norm{\sqrt{\rho}\dot{u}}_{L^2}^2+C(\tilde{\rho})\norm{\rho^{\frac{1}{3}}u}_{L^3}(\norm{\nabla u}_{L^2}^2+\norm{P}_{L^1}^{\frac{1}{3}})+C(\tilde{\rho})(\norm{\nabla u}_{L^2}^2+\norm{P}_{L^1}),
  \end{aligned}
  \end{equation}
  where we have used \eqref{Pressure-eq}, \eqref{a-priori-assumption}, \eqref{curl-effective-viscous-estimate2} and \eqref{curl-effective-viscous-estimate6}.
  
  Then, integrating \eqref{small-T-estimate1} on $[0,\sigma(T)]$ and using \eqref{curl-effective-viscous-estimate1}, \eqref{essential-energy} and \eqref{essential-estimate}, we have
  \begin{equation}\label{small-T-estimate2}
    \begin{aligned}
    &\quad\sup_{t\in[0,\sigma(T)]}\norm{\nabla u}_{L^2}^2+\int_0^{\sigma(T)}\int\rho|\dot{u}|^2\\
    &\leq C(M)+C(\tilde{\rho})\mathcal{E}_0+C(\tilde{\rho})A_3^{\frac{1}{3}}(\sigma(T))(\mathcal{E}_0+\mathcal{E}_0^{\frac{1}{3}})\leq C(\tilde{\rho},M)
    \end{aligned}
  \end{equation}
  provided
  \begin{equation*}
    C(\tilde{\rho})A_3^{\frac{1}{3}}(\sigma(T))\leq\frac{1}{4},\quad \textrm{i.e.}, \quad \mathcal{E}_0\leq(4C(\tilde{\rho}))^{-12}.
  \end{equation*}
  Thus, we complete the proof of \eqref{small-time-estimate1}.
  
  Next, we turn to prove \eqref{small-time-estimate2}. Taking $m=1$ and $T=\sigma(T)$ in \eqref{A-control-4}, we have from \eqref{small-time-estimate1} that
  \begin{equation}\label{small-T-estimate3}
    \begin{aligned}
    &\quad\sup_{t\in[0,\sigma(T)]}\sigma\int\rho|\dot{u}|^2+\int_0^{\sigma(T)}\int\sigma|\nabla\dot{u}|^2\\
    &\leq \sup_{t\in[0,\sigma(T)]}\sigma(C\norm{\nabla u}_{L^2}^3+C(\tilde{\rho})\norm{\nabla u}_{L^2}^4)+C\int_0^{\sigma(T)}(\norm{\nabla u}_{L^2}^2(\norm{\sqrt{\rho}\dot{u}}_{L^2}^2+C(\tilde{\rho}))+\norm{\sqrt{\rho}\dot{u}}_{L^2}^2+\norm{\nabla u}_{L^2}^4)\\
    &\quad+C\int_0^{\sigma(T)}\sigma(\norm{\nabla u}_{L^4}^4+\norm{P|\nabla u|}_{L^2}^2+\norm{P-\bar{P}}_{L^4}^4)+C(\tilde{\rho})\int_0^{\sigma(T)}\sigma(\norm{\nabla u}_{L^2}^8+\norm{\sqrt{\rho}\dot{u}}_{L^2}^2\norm{\nabla u}_{L^2}^4)\\
    &\leq C(\tilde{\rho},M)A_1(\sigma(T))+C(\tilde{\rho},M)+C(\tilde{\rho},M)\mathcal{E}_0+C(\tilde{\rho},M)A_1(\sigma(T))\mathcal{E}_0+C(\tilde{\rho},M)A_1(\sigma(T))\\
    &\quad+\frac{1}{2}\sup_{t\in[0,\sigma(T)]}\sigma\norm{\sqrt{\rho}\dot{u}}_{L^2}^2,
    \end{aligned}
  \end{equation}
  where we have used the following estimate
  \begin{equation*}
    \begin{aligned}
    C\int_0^{\sigma(T)}\sigma\norm{\nabla u}_{L^4}^4&\leq C\int_0^{\sigma(T)}\sigma\norm{\nabla u}_{L^2}\norm{\nabla u}_{L^6}^3\\
    &\leq C\int_0^{\sigma(T)}\sigma\norm{\nabla u}_{L^2}(\norm{\rho\dot{u}}_{L^2}+\norm{\nabla u}_{L^2}+\norm{P-\bar{P}}_{L^6})^3\\
    &\leq C(\tilde{\rho},M)\sup_{t\in[0,\sigma(T)]}\sigma\norm{\sqrt{\rho}\dot{u}}_{L^2}\int_0^{\sigma(T)}\norm{\sqrt{\rho}\dot{u}}_{L^2}^2+C(\tilde{\rho})A_1(\sigma(T))\mathcal{E}_0+C(\tilde{\rho})\mathcal{E}_0\\
    &\leq\frac{1}{2}\sup_{t\in[0,\sigma(T)]}\sigma\norm{\sqrt{\rho}\dot{u}}_{L^2}^2+C(\tilde{\rho},M)+C(\tilde{\rho})A_1(\sigma(T))\mathcal{E}_0+C(\tilde{\rho})\mathcal{E}_0
    \end{aligned}
  \end{equation*}
  due to \eqref{curl-effective-viscous-estimate6}, \eqref{essential-energy} and \eqref{essential-estimate}.
  Then, \eqref{small-T-estimate3} implies
  \begin{equation}\label{small-T-estimate4}
    \sup_{t\in[0,\sigma(T)]}\sigma\int\rho|\dot{u}|^2+\int_0^{\sigma(T)}\int\sigma|\nabla\dot{u}|^2\leq C(\tilde{\rho},M),
  \end{equation} immediately.
  Therefore, we complete the proof of \eqref{small-time-estimate2}.
  The estimates \eqref{small-T-estimate3} and \eqref{small-T-estimate4} are the combinations of \eqref{small-T-estimate1}-\eqref{small-time-estimate2} and assumption \eqref{a-priori-assumption}.
  Thus, we have completed the proof of Lemma \ref{lem-small-time-estimate}
\end{proof}

\begin{lem}\label{lem-u-L^3}
  Under the conditions of Proposition \ref{prop-a-priori-estimate}, it holds that
  \begin{equation}\label{u-L^3-goal}
    A_3(\sigma(T))\leq\mathcal{E}_0^{\frac{1}{4}},
  \end{equation}
  provided
  \begin{equation}\label{small-assumption2}
    \mathcal{E}_0\leq\epsilon_3=\min\{\epsilon_2,(C(\tilde{\rho},M))^{-2}\}.
  \end{equation}
\end{lem}

\begin{proof}
  Multiplying $\eqref{Large-CNS-eq}_2$ by $3|u|u$ and integrating over $\Omega$ yields that
  \begin{equation}\label{u-L^3-estimate1}
    \begin{aligned}
    &\quad\frac{d}{dt}\int\rho|u|^3+3(2\mu+\lambda)\int\div u\div(u|u|)+3\mu\int\curl u\cdot\curl(u|u|)+3\mu\int_{\pa\Omega}Au\cdot u|u|\\
    &=3\int P\div(u|u|)
    \end{aligned}
  \end{equation}
  which together with \eqref{curl-effective-viscous-estimate6} implies that
  \begin{equation}\label{u-L^3-estimate2}
    \begin{aligned}
    &\quad\frac{d}{dt}\int\rho|u|^3+3(2\mu+\lambda)\int|\div u|^2|u|+3\mu\int|\curl u|^2|u|+3\mu\int_{\pa\Omega}Au\cdot u|u|\\
    &\leq C\int|u||\nabla u|^2+C\int P|u||\nabla u|\\
    &\leq C\norm{u}_{L^6}\norm{\nabla u}_{L^2}^{\frac{3}{2}}\norm{\nabla u}_{L^6}+C\norm{u}_{L^6}\norm{\nabla u}_{L^2}\norm{P}_{L^3}\\
    &\leq C\norm{\nabla u}_{L^2}^{\frac{5}{2}}(\norm{\rho\dot{u}}_{L^2}+\norm{\nabla u}_{L^2}+\norm{P-\bar{P}}_{L^6})^{\frac{1}{2}}+C\norm{\nabla u}_{L^2}^2\norm{P}_{L^3}. 
    \end{aligned}
  \end{equation}
  Then integrating over $[0,\sigma(T)]$, we have from \eqref{small-time-estimate1}, \eqref{small-time-estimate2}, \eqref{a-priori-assumption} and \eqref{essential-estimate}
  that
  \begin{equation}\label{u-L^3-estimate3}
    \begin{aligned}
    A_3(\sigma(T))&\leq C(\tilde{\rho})\sup_{t\in[0,\sigma(T)]}\norm{\nabla u}_{L^2}(\int_0^{\sigma(T)}\norm{\nabla u}_{L^2}^2)^{\frac{3}{4}}(\int_0^{\sigma(T)}\norm{\sqrt{\rho}\dot{u}}_{L^2}^2)^{\frac{1}{4}}\\
    &\quad+\int\rho_0|u_0|^3+C(\tilde{\rho},M)\mathcal{E}_0\\
    &\leq C(\tilde{\rho},M)\mathcal{E}_0^{\frac{3}{4}}++C(\tilde{\rho},M)\mathcal{E}_0
    +C(\tilde{\rho})\norm{\sqrt{\rho_0}u_0}_{L^2}^{\frac{3}{2}}\norm{\nabla u_0}_{L^2}^{\frac{3}{2}}\\
    &\leq C(\tilde{\rho},M)\mathcal{E}_0^{\frac{3}{4}}\leq \mathcal{E}_0^{\frac{1}{4}}    
    \end{aligned}
  \end{equation}
  provided
  \begin{equation*}
    C(\tilde{\rho},M)\mathcal{E}_0^{\frac{1}{2}}\leq1,\quad\textrm{i.e.},\quad\mathcal{E}_0\leq(C(\tilde{\rho},M))^{-2}.
  \end{equation*}
  Thus, we prove \eqref{u-L^3-goal} and thus complete the proof of \ref{lem-u-L^3}.
\end{proof}

With the following lemma, we can complete the proof of Proposition \ref{prop-a-priori-estimate}.
\begin{lem}\label{lem-control-A1-A2}
  Under the conditions of Proposition \ref{prop-a-priori-estimate}, there holds that
  \begin{equation}\label{control-A1-A2}
    A_1(T)\leq\mathcal{E}_0^{\frac{3}{8}},\quad A_1(T)\leq\mathcal{E}_0^{\frac{1}{2}},
  \end{equation}
  provided
  \begin{equation}\label{small-assumption3}
    \mathcal{E}_0\leq\epsilon_4=\min\{\epsilon_3,(2C(\tilde{\rho}))^{-16},(2C(\tilde{\rho},M)(E_0+1))^{-2},(2C(\tilde{\rho})E_0^{\frac{1}{2}})^{-32},(2C(\tilde{\rho},M))^{-\frac{16}{9}}\}
  \end{equation}
  and
  \begin{equation}\label{small-A-assumption1}
    \norm{A}_{W^{1,6}}\leq\min\left\{(2C)^{-\frac{1}{3}}, (27C(\Omega))^{-\frac{1}{3}}, (2CC(\Omega))^{-\frac{1}{3}}\right\}
  \end{equation}
  with $C$ here depending only on $\mu,\lambda$ and $\Omega$, and $C(\Omega)$ only depending on $\Omega$.
\end{lem}

\begin{proof}
 Due to the lack of smallness of $\displaystyle\int_0^T\int|\nabla u|^2$, motivated by \cite{Zhu2024}, we first estimate $\displaystyle\int_0^T\sigma^3\norm{P-\bar{P}}_{L^4}^4$ rather than $\displaystyle\int_0^T\int|P-\bar{P}|^2$ in \cite{Cai-Li2023}. To begin with, it follows from $\eqref{Large-CNS-eq}_1$ that
  \begin{equation}\label{Pressure-eq2}
    (P-\bar{P})_t+u\cdot\nabla(P-\bar{P})+\gamma(P-\bar{P})\div u+\gamma\bar{P}\div u-(\gamma-1)\overline{P\div u}=0.
  \end{equation}
  Multiplying \eqref{Pressure-eq2} by $3\sigma^3(P-\bar{P})^2$ and integrating the resultant equation over $\Omega\times[0,T]$, we get from the fact that $(2\mu+\lambda)\div u=G+(P-\bar{P})$ that
  \begin{equation}\label{P-L^4-estimate1}
    \begin{aligned}
    &\quad\frac{d}{dt}\int\sigma^3(P-\bar{P})^3+\frac{3\gamma-1}{2\mu+\lambda}\int\sigma^3|P-\bar{P}|^4\\
    &=3\sigma^2\sigma'\int(P-\bar{P})^3-\frac{3\gamma-1}{2\mu+\lambda}\int\sigma^3(P-\bar{P})^3G+3(\gamma-1)\overline{P\div u}\sigma^3\int(P-\bar{P})^2\\
    &\quad-3\gamma\bar{P}\int\sigma^3(P-\bar{P})^2\div u\\
    &\leq3\sigma^2\sigma'\int(P-\bar{P})^3+\frac{3\gamma-1}{2(2\mu+\lambda)}\int\sigma^3|P-\bar{P}|^4+\frac{27}{4}\frac{3\gamma-1}{2\mu+\lambda}\sigma^3\norm{G}_{L^4}^4\\
    &\quad+C(\gamma-1)\sigma^3\norm{P}_{L^2}^2\norm{\nabla u}_{L^2}^2+C\gamma\sigma^3\bar{P}^2\norm{\nabla u}_{L^2}^2,
    \end{aligned}
  \end{equation}
  which together with the following simple fact by \eqref{curl-G-modified-estimate3}, \eqref{a-priori-assumption} and \eqref{essential-energy}
  \begin{equation}\label{G-L^4-control1}
    \begin{aligned}
    \int_0^T\sigma^3\norm{G}_{L^4}^4&\leq \int_0^T\sigma^3\norm{G}_{L^2}\norm{G}_{L^6}^3\\
    &\leq C(\Omega)\int_0^T\sigma^3(\norm{\nabla u}_{L^2}+\norm{P-\bar{P}}_{L^2})(\norm{\rho\dot{u}}_{L^2}^3+\norm{A}_{W^{1,6}}^3\norm{\nabla u}_{L^2}^3)\\
    &\leq C(\Omega)\int_0^T\sigma^3(\norm{\nabla u}_{L^2}+\norm{P-\bar{P}}_{L^2})\norm{\rho\dot{u}}_{L^2}^3+C(\Omega)\norm{A}_{W^{1,6}}^3\int_0^T\sigma^3\norm{P-\bar{P}}_{L^4}^4\\
    &\quad+C(\Omega)\norm{A}_{W^{1,6}}^3\int_0^T\sigma^3\norm{\nabla u}_{L^4}^4\\
    &\leq C(\tilde{\rho})(A_1^{\frac{1}{2}}(T)+\mathcal{E}_0^{\frac{1}{2}})A_1(T)A_2^{\frac{1}{2}}(T)+C(\Omega)\norm{A}_{W^{1,6}}^3\int_0^T\sigma^3\norm{P-\bar{P}}_{L^4}^4\\
    &\quad+C(\Omega)\norm{A}_{W^{1,6}}^3\int_0^T\sigma^3\norm{\nabla u}_{L^4}^4
    \end{aligned}
  \end{equation}
  and inequalities \eqref{a-priori-assumption} and \eqref{essential-energy} yields that
  \begin{equation}\label{P-L^4-estimate2}
    \begin{aligned}
    \int_0^T\int\sigma^3|P-\bar{P}|^4&\leq C(\tilde{\rho})\mathcal{E}_0+C(\tilde{\rho})A_1^{\frac{3}{2}}(T)A_2^{\frac{1}{2}}(T)
    +C(\tilde{\rho})\mathcal{E}_0^2+C\mathcal{E}_0^2E_0\\
    &\quad+C(\Omega)\norm{A}_{W^{1,6}}^3\int_0^T\sigma^3\norm{\nabla u}_{L^4}^4
    \end{aligned}
  \end{equation}
  provided
  \begin{equation}\label{small-A-condition1}
    \frac{27}{4}C(\Omega)\norm{A}_{W^{1,6}}^3\leq\frac{1}{4},\quad\textrm{i.e.},\quad
    \norm{A}_{W^{1,6}}\leq(27C(\Omega))^{-\frac{1}{3}}
  \end{equation}
  with $C(\Omega)$ depending on $\Omega$ only.
  
  Then, for $A_2(T)$, by virtue of \eqref{P-L^4-estimate2} and \eqref{small-A-condition1}, we have
  \begin{equation}\label{A2-estimate-1}
    \begin{aligned}
    &\quad\int_0^T\sigma^3(\norm{\nabla u}_{L^4}^4+\norm{P|\nabla u|}_{L^2}^2+\norm{P-\bar{P}}_{L^4}^4)\\
    &\leq C\int_0^T\sigma^3\norm{\nabla u}_{L^4}^4+C\int_0^T\sigma^3\norm{P-\bar{P}}_{L^4}^4+C\int_0^T\bar{P}^2\sigma^3\norm{\nabla u}_{L^2}^2\\
    &\leq C\int_0^T\sigma^3\norm{\nabla u}_{L^4}^4+C(\tilde{\rho})\mathcal{E}_0+C(\tilde{\rho})A_1^{\frac{3}{2}}(T)A_2^{\frac{1}{2}}(T)
    +C(\tilde{\rho})\mathcal{E}_0^2+C\mathcal{E}_0^2E_0.  
    \end{aligned}
  \end{equation}
  
  Next, we aim to control $\int_0^T\sigma^3\norm{\nabla u}_{L^4}^4$. If certain smallness is imposed on $A$ (see \eqref{small-A-condition}), due to the discussion in \eqref{nabla-u-L^4-estimate2} and the estimate \eqref{P-L^4-estimate2}, we finally give the control of $\displaystyle\int_0^T\sigma^3\norm{\nabla u}_{L^4}^4$ as follows:
  \begin{equation}\label{A2-estimate-2}
    \begin{aligned}
    \int_0^T\sigma^3\norm{\nabla u}_{L^4}^4&\leq C\int_0^T\sigma^3(\norm{\nabla u}_{L^2}+\norm{P-\bar{P}}_{L^2})\norm{\rho\dot{u}}_{L^2}^3+C\int_0^T\sigma^3\norm{P-\bar{P}}_{L^4}^4\\
    &\leq C(\tilde{\rho})A_1^{\frac{3}{2}}(T)A_2^{\frac{1}{2}}(T)
    +C(\tilde{\rho})\mathcal{E}_0+C(\tilde{\rho})\mathcal{E}_0^2+C\mathcal{E}_0^2E_0\\
    &\quad+CC(\Omega)\norm{A}_{W^{1,6}}^3\int_0^T\sigma^3\norm{\nabla u}_{L^4}^4,\\
    &\quad(\textrm{with $C$ only depending on $\mu, \lambda$ and $\Omega$ here})
    \end{aligned}
  \end{equation}
  which implies that
  \begin{equation}\label{A2-estimate-3}
    \begin{aligned}
      \int_0^T\sigma^3\norm{\nabla u}_{L^4}^4&\leq C(\tilde{\rho})A_1^{\frac{3}{2}}(T)A_2^{\frac{1}{2}}(T)
      +C(\tilde{\rho})\mathcal{E}_0+C(\tilde{\rho})\mathcal{E}_0^2+C\mathcal{E}_0^2E_0\\
      &\leq C(\tilde{\rho})A_1^{\frac{3}{2}}(T)A_2^{\frac{1}{2}}(T)+C\mathcal{E}_0^2E_0,
    \end{aligned}  
  \end{equation}
  provided
  \begin{equation*}
    \norm{A}_{W^{1,6}}\leq\min\left\{(2C)^{-\frac{1}{3}}, (27C(\Omega))^{-\frac{1}{3}}, (2CC(\Omega))^{-\frac{1}{3}}\right\}
  \end{equation*}
  with $C$ here depending only on $\mu,\lambda$ and $\Omega$, and $C(\Omega)$ only depending on $\Omega$.
  
  Then inserting \eqref{A2-estimate-1} and \eqref{A2-estimate-3} into \eqref{A2-control}, we obtain
  \begin{equation}\label{A2-estimate-4}
    \begin{aligned}
    A_2(T)&\leq C(\tilde{\rho})\mathcal{E}_0+CA_1^{\frac{3}{2}}(T)(1+C(\tilde{\rho})A_1(T))^{\frac{1}{2}}+CA_1(\sigma(T))+C(\tilde{\rho})A_1^3(T)(E_0+1)\\
     &\quad+C(\tilde{\rho})A_1^{\frac{3}{2}}(T)A_2^{\frac{1}{2}}(T)
    +C(\tilde{\rho})\mathcal{E}_0^2+C\mathcal{E}_0^2E_0\\
    &\leq C(\tilde{\rho})A_1^{\frac{3}{2}}(T)+CA_1(\sigma(T))+C(\tilde{\rho})A_1^3(T)(1+E_0). 
    \end{aligned}
  \end{equation} 
  Recalling \eqref{A1-control}, we have
  \begin{equation}\label{A1-estimate-1}
    \begin{aligned}
      A_1(T)&\leq C(\tilde{\rho})\mathcal{E}_0+C\mathcal{E}_0E_0+C\int_0^T\int\sigma P|\nabla u|^2+C\int_0^T\sigma\norm{\nabla u}_{L^3}^3+C(\tilde{\rho})\int_0^T\sigma\norm{\nabla u}_{L^2}^4\\
      &\leq C(\tilde{\rho})\mathcal{E}_0(1+E_0)+C\int_0^T\int\sigma |P-\bar{P}||\nabla u|^2+C\bar{P}\int_0^T\int\sigma|\nabla u|^2+C\int_0^T\sigma\norm{\nabla u}_{L^3}^3\\
      &\quad+C(\tilde{\rho})\int_0^T\sigma\norm{\nabla u}_{L^2}^4\\
      &\leq C(\tilde{\rho})\mathcal{E}_0(1+E_0)+C\int_0^T\int\sigma |P-\bar{P}||\nabla u|^2+C\int_0^T\sigma\norm{\nabla u}_{L^3}^3+C(\tilde{\rho})\int_0^T\sigma\norm{\nabla u}_{L^2}^4.
    \end{aligned}
  \end{equation}
  Then from \eqref{essential-estimate}, \eqref{curl-effective-viscous-estimate6}, \eqref{essential-energy}, \eqref{a-priori-assumption} and \eqref{small-time-estimate2}, it holds that
  \begin{equation}\label{A1-estimate-small-T1}
    \begin{aligned}
    &\quad A_1(\sigma(T))\\
    &\leq C(\tilde{\rho})\mathcal{E}_0(1+E_0)+C\int_0^{\sigma(T)}\int\sigma |P-\bar{P}||\nabla u|^2+C\int_0^{\sigma(T)}\sigma\norm{\nabla u}_{L^3}^3\\
    &\quad+C(\tilde{\rho})\int_0^{\sigma(T)}\sigma\norm{\nabla u}_{L^2}^4\\
    &\leq C(\tilde{\rho})\mathcal{E}_0(1+E_0)+C(\tilde{\rho})\mathcal{E}_0(1+A_1(\sigma(T)))+C\int_0^{\sigma(T)}\sigma\norm{\nabla u}_{L^2}^{\frac{3}{2}}\norm{\nabla u}_{L^6}^{\frac{3}{2}}\\
    &\leq C(\tilde{\rho})\mathcal{E}_0(1+E_0)+C\int_0^{\sigma(T)}\sigma\norm{\nabla u}_{L^2}^{\frac{3}{2}}(\norm{\rho\dot{u}}_{L^2}+\norm{\nabla u}_{L^2}+\norm{P-\bar{P}}_{L^6})^{\frac{3}{2}}\\
    &\leq C(\tilde{\rho})\mathcal{E}_0(1+E_0)+C(\tilde{\rho})\sup_{t\in[0,\sigma(T)]}\sigma^{\frac{1}{2}}\norm{\sqrt{\rho}\dot{u}}_{L^2}(\int_0^{\sigma(T)}\norm{\nabla u}_{L^2}^2)^{\frac{3}{4}}(\int_0^{\sigma(T)}\sigma^2\norm{\sqrt{\rho}\dot{u}}_{L^2}^2)^{\frac{1}{4}}\\
    &\quad+C(\tilde{\rho})A_1^{\frac{1}{2}}(\sigma(T))\mathcal{E}_0+C(\int_0^{\sigma(T)}\norm{\nabla u}_{L^2}^2)^{\frac{3}{4}}(\int_0^{\sigma(T)}\norm{P-\bar{P}}_{L^6}^6)^{\frac{1}{4}}\\
    &\leq C(\tilde{\rho})\mathcal{E}_0(1+E_0)+C(\tilde{\rho},M)\mathcal{E}_0^{\frac{3}{4}}A_1^{\frac{1}{4}}(\sigma(T))+C(\tilde{\rho})\mathcal{E}_0\\
    &\leq C(\tilde{\rho})\mathcal{E}_0(1+E_0)+C(\tilde{\rho},M)\mathcal{E}_0+\frac{1}{2}A_1(\sigma(T)),    
    \end{aligned}
  \end{equation}
  which implies
  \begin{equation}\label{A1-estimate-small-T2}
    A_1(\sigma(T))\leq C(\tilde{\rho},M)\mathcal{E}_0(1+E_0).
  \end{equation}
  Then we turn back to \eqref{A2-estimate-4} and get
  \begin{equation}\label{A2-estimate-5}
    \begin{aligned}
    A_2(T)&\leq C(\tilde{\rho})A_1^{\frac{3}{2}}(T)+C(\tilde{\rho},M)\mathcal{E}_0(1+E_0)+C(\tilde{\rho})A_1^3(T)(1+E_0)\\
    &\leq C(\tilde{\rho})A_1^{\frac{3}{2}}(T)+C(\tilde{\rho},M)\mathcal{E}_0(1+E_0)\\
    &\leq C(\tilde{\rho})\mathcal{E}_0^{\frac{9}{16}}+C(\tilde{\rho},M)\mathcal{E}_0(1+E_0)\\
    &\leq\mathcal{E}_0^{\frac{1}{2}} 
    \end{aligned}
  \end{equation}
  provided
  \begin{equation}\label{small-condition-transition1}
    C(\tilde{\rho})\mathcal{E}_0^{\frac{1}{16}}\leq\frac{1}{2},\,C(\tilde{\rho},M)\mathcal{E}_0^{\frac{1}{2}}(1+E_0)\leq\frac{1}{2},\,
    \textrm{i.e.},\,\mathcal{E}_0\leq\min\{(2C(\tilde{\rho}))^{-16},(2C(\tilde{\rho},M)(E_0+1))^{-2}\}.
  \end{equation}
  Thus we finish the estimate on $A_2(T)$. 
  
  It is easy to check that under the condition \eqref{small-condition-transition1},
  \begin{equation}\label{A1-estimate-small-T3}
    A_1(\sigma(T))\leq\mathcal{E}_0^{\frac{1}{2}}\leq\mathcal{E}_0^{\frac{3}{8}}
  \end{equation}
  and also by \eqref{A2-estimate-3} and \eqref{P-L^4-estimate2},
  \begin{equation}\label{nabla-u-P-L^4-estimate}
    \int_0^T\sigma^3\norm{\nabla u}_{L^4}^4\leq C(\tilde{\rho})A_1^{\frac{3}{2}}(T)A_2^{\frac{1}{2}}(T),\quad \int_0^T\int\sigma^3|P-\bar{P}|^4\leq C(\tilde{\rho})A_1^{\frac{3}{2}}(T)A_2^{\frac{1}{2}}(T).
  \end{equation}
  To estimate $A_1(T)$, it suffices to control $\displaystyle\int_{\sigma(T)}^T\int\sigma |P-\bar{P}||\nabla u|^2$ and $\displaystyle\int_{\sigma(T)}^T\sigma\norm{\nabla u}_{L^3}^3$. Due to \eqref{a-priori-assumption}, we have
  \begin{equation}\label{A1-estimate-2}
    C(\tilde{\rho})\int_{\sigma(T)}^T\sigma\norm{\nabla u}_{L^2}^4\leq C(\tilde{\rho})A_1^{\frac{1}{2}}(T)\int_{\sigma(T)}^T\sigma\norm{\nabla u}_{L^3}^3.
  \end{equation}
  From \eqref{nabla-u-P-L^4-estimate}, \eqref{essential-energy} and \eqref{a-priori-assumption}, we obtain\footnote{During this calculation, we observe that $A_1(T)^{\frac{3}{4}}A_2^{\frac{1}{4}}(T)E_0^{\frac{1}{2}}\ll A_1(T)$ needs $A_2(T)\ll A_1(T)$. And in \eqref{A2-estimate-4} we also need $A^{\frac{3}{2}}(T)\ll A_2(T)$. Thus we can ensure the setting of a priori assumption \eqref{a-priori-assumption}.}
  \begin{equation}\label{A1-estimate-3}
    \begin{aligned}
    &\quad\int_{\sigma(T)}^T\int\sigma |P-\bar{P}||\nabla u|^2\\
    &\leq\left(\int_{\sigma(T)}^T\norm{P-\bar{P}}_{L^4}^4\right)^{\frac{1}{4}}\left(\int_{\sigma(T)}^T\norm{\nabla u}_{L^4}^4\right)^{\frac{1}{4}}\left(\int_{\sigma(T)}^T\norm{\nabla u}_{L^2}^2\right)^{\frac{1}{2}}\\
    &\leq C(\tilde{\rho})A_1(T)^{\frac{3}{4}}A_2^{\frac{1}{4}}(T)E_0^{\frac{1}{2}}\\
    &\leq C(\tilde{\rho})\mathcal{E}_0^{\frac{13}{32}}E_0^{\frac{1}{2}}.
    \end{aligned}
  \end{equation}
  Similarly, it holds from \eqref{nabla-u-P-L^4-estimate}, \eqref{essential-energy} and \eqref{a-priori-assumption} that
  \begin{equation}\label{A1-estimate-4}
    \begin{aligned}
    \int_{\sigma(T)}^T\sigma\norm{\nabla u}_{L^3}^3&\leq\int_{\sigma(T)}^T\norm{\nabla u}_{L^2}\norm{\nabla u}_{L^4}^2\leq\left(\int_{\sigma(T)}^T\norm{\nabla u}_{L^2}^2\right)^{\frac{1}{2}}\left(\int_{\sigma(T)}^T\norm{\nabla u}_{L^4}^4\right)^{\frac{1}{2}}\\
    &\leq C(\tilde{\rho})\mathcal{E}_0^{\frac{13}{32}}E_0^{\frac{1}{2}}.
    \end{aligned}
  \end{equation}
  Then, plugging \eqref{A1-estimate-2}-\eqref{A1-estimate-4}, \eqref{A1-estimate-small-T2} and \eqref{a-priori-assumption} into \eqref{A1-estimate-1} yields that
  \begin{equation}\label{A1-estimate-5}
    \begin{aligned}
    A_1(T)&\leq A_1(\sigma(T))+C\int_{\sigma(T)}^T\int\sigma |P-\bar{P}||\nabla u|^2+C\int_{\sigma(T)}^T\sigma\norm{\nabla u}_{L^3}^3+C(\tilde{\rho})\int_{\sigma(T)}^T\sigma\norm{\nabla u}_{L^2}^4\\
    &\leq C(\tilde{\rho},M)\mathcal{E}_0(1+E_0)+C(\tilde{\rho})\mathcal{E}_0^{\frac{13}{32}}E_0^{\frac{1}{2}}(1+A_1^{\frac{1}{2}}(T))\\
    &\leq C(\tilde{\rho},M)\mathcal{E}_0(1+E_0)+C(\tilde{\rho})\mathcal{E}_0^{\frac{13}{32}}E_0^{\frac{1}{2}}\\
    &\leq \frac{1}{2}\mathcal{E}_0^{\frac{3}{8}}+C(\tilde{\rho},M)\mathcal{E}_0^{\frac{15}{16}}\leq\mathcal{E}_0^{\frac{3}{8}} 
    \end{aligned}
  \end{equation}
  provided
  \begin{equation*}
    \mathcal{E}_0\leq\min\{(2C(\tilde{\rho})E_0^{\frac{1}{2}})^{-32},(2C(\tilde{\rho},M))^{-\frac{16}{9}}\}.
  \end{equation*}
  Thus, we have completed the proof of Lemma \ref{lem-control-A1-A2}. 
\end{proof}

\begin{rem}
  Note that in this lemma, the condition $\norm{\nabla u_0}_{L^2}\leq M$ is only applied in \eqref{A1-estimate-small-T1} through Lemma \ref{lem-small-time-estimate}.
\end{rem}

Next lemma closes the estimate on bound of density $\rho$.

\begin{lem}\label{lem-rho-bound}
  Under the conditions of Proposition \ref{prop-a-priori-estimate}, it holds that for any $(x,t)\in\Omega\times[0,T]$,
  \begin{equation}\label{rho-bound}
    0\leq\rho(x,t)\leq\frac{7\tilde{\rho}}{4},
  \end{equation}
  provided
  \begin{equation}\label{small-assumption4}
    \mathcal{E}_0\leq\epsilon_5=\min\left\{\epsilon_4,(\frac{\tilde{\rho}}{2C(\tilde{\rho},M)})^{-12},(C(\tilde{\rho}))^{-1},(\frac{\tilde{\rho}}{4C(\tilde{\rho})(1+E_0)})^3\right\}.
  \end{equation}
\end{lem}

\begin{proof}
  First, we rewrite the equation of mass conservation $\eqref{Large-CNS-eq}_1$ as
  \begin{equation}\label{mass-equation}
    D_t\rho=g(\rho)+b'(t),
  \end{equation}
  where
  \begin{equation*}
    D_t\rho=\rho_t+u\cdot\nabla\rho,\quad g(\rho)=-\frac{\rho P}{2\mu+\lambda},\quad b(t)=\frac{1}{2\mu+\lambda}\int_0^t(\rho\bar{P}-\rho G).
  \end{equation*}
  Then for $t\in[0,\sigma(T)]$, we obtain from \eqref{essential-energy}, \eqref{a-priori-assumption}, \eqref{GN-inequality2}, \eqref{A1-estimate-small-T3}, Lemma \ref{lem-curl-effective-viscous}, Lemma \ref{lem-dot-u} and Lemma \ref{lem-small-time-estimate} that for any $0\leq t_1<t_2\leq\sigma(T)$,
  \begin{equation}\label{b-estimate1}
    \begin{aligned}
    &\quad|b(t_2)-b(t_1)|\\
    &\leq C(\tilde{\rho})\mathcal{E}_0+C(\tilde{\rho})\int_0^{\sigma(T)}\norm{G}_{L^{\infty}}\\
    &\leq C(\tilde{\rho})\mathcal{E}_0+C(\tilde{\rho})\int_0^{\sigma(T)}\norm{G}_{L^6}^{\frac{1}{2}}\norm{\nabla G}_{L^6}^{\frac{1}{2}}\\
    &\leq C(\tilde{\rho})\mathcal{E}_0+C(\tilde{\rho})\int_0^{\sigma(T)}(\norm{\sqrt{\rho}\dot{u}}_{L^2}+\norm{\nabla u}_{L^2})^{\frac{1}{2}}(\norm{\nabla\dot{u}}_{L^2}+\norm{\nabla u}_{L^2}^2+\norm{\nabla u}_{L^2}+\norm{P-\bar{P}}_{L^6})^{\frac{1}{2}}\\
    &\leq C(\tilde{\rho})\mathcal{E}_0+C(\tilde{\rho})\left(\int_0^{\sigma(T)}t\norm{\nabla\dot{u}}_{L^2}^2\right)^{\frac{1}{4}}
    \left(\int_0^{\sigma(T)}t^{-\frac{1}{3}}\norm{\sqrt{\rho}\dot{u}}_{L^2}^{\frac{2}{3}}\right)^{\frac{3}{4}}\\
    &\quad+C(\tilde{\rho})A_1^{\frac{1}{4}}(\sigma(T))\int_0^{\sigma(T)}t^{-\frac{1}{2}}(t\norm{\nabla\dot{u}}_{L^2}^2)^{\frac{1}{4}}
    +C(\tilde{\rho},M)\int_0^{\sigma(T)}t^{-\frac{1}{4}}(\norm{\nabla u}_{L^2}+\norm{\nabla u}_{L^2}^{\frac{1}{2}}+\mathcal{E}_0^{\frac{1}{12}})\\
    &\leq C(\tilde{\rho})\mathcal{E}_0+C(\tilde{\rho},M)\left(\int_0^{\sigma(T)}t^{-\frac{2}{3}}(t\norm{\sqrt{\rho}\dot{u}}_{L^2}^2)^{\frac{1}{4}}\right)^{\frac{3}{4}}
    +C(\tilde{\rho},M)A_1^{\frac{1}{4}}(\sigma(T))+C(\tilde{\rho},M)\mathcal{E}_0^{\frac{1}{12}}\\
    &\leq C(\tilde{\rho},M)(A_1^{\frac{3}{16}}(\sigma(T))+\mathcal{E}_0^{\frac{1}{12}})\\
    &\leq C(\tilde{\rho},M)(\mathcal{E}_0^{\frac{3}{32}}+\mathcal{E}_0^{\frac{1}{12}})\leq C(\tilde{\rho},M)\mathcal{E}_0^{\frac{1}{12}}
    \end{aligned}
  \end{equation}
   provided $\mathcal{E}_0\leq\epsilon_4$. Then, by choosing $N_1=0$, $N_0=C(\tilde{\rho},M)\mathcal{E}_0^{\frac{1}{12}}$, and $\zeta_0=\tilde{\rho}$ in Lemma \ref{lem-Zlotnik-inequality}, we have from \eqref{mass-equation} and \eqref{b-estimate1} that
   \begin{equation}\label{rho-bound-small-T}
     \sup_{t\in[0,\sigma(T)]}\norm{\rho}_{L^{\infty}}\leq\tilde{\rho}+C(\tilde{\rho},M)\mathcal{E}_0^{\frac{1}{12}}\leq\frac{3\tilde{\rho}}{2}
   \end{equation}
   provided $\mathcal{E}_0\leq\min\left\{\epsilon_4,(\frac{\tilde{\rho}}{2C(\tilde{\rho},M)})^{-12}\right\}$.
   
   For $t\in[\sigma(T),T]$ and any $\sigma(T)\leq t_1<t_2\leq T$, we also have
   \begin{equation}\label{b-estimate2}
     \begin{aligned}
     |b(t_2)-b(t_1)|&\leq C\tilde{\rho}\int_{t_1}^{t_2}\norm{G}_{L^{\infty}}+C\int_{t_1}^{t_2}\rho\bar{P}\\
     &\leq \frac{(1+C(\tilde{\rho})\mathcal{E}_0)\tilde{\rho}P(\tilde{\rho})}{2(2\mu+\lambda)}(t_2-t_1)+C(\tilde{\rho})\int_{\sigma(T)}^{T}\norm{G}_{L^{\infty}}^4\\
     &\leq \frac{\tilde{\rho}P(\tilde{\rho})}{2\mu+\lambda}(t_2-t_1)+C(\tilde{\rho})\mathcal{E}_0^{\frac{1}{3}}(1+E_0), 
     \end{aligned}
   \end{equation}
   where we have used that by Lemma \ref{lem-curl-effective-viscous}, Lemma \ref{lem-dot-u}, \eqref{essential-energy} and \eqref{a-priori-assumption},
   \begin{equation*}
     \begin{aligned}
     &\quad\int_{\sigma(T)}^{T}\norm{G}_{L^{\infty}}^4
     \leq\int_{\sigma(T)}^{T}\norm{G}_{L^6}^2\norm{\nabla G}_{L^6}^2\\
     &\leq C\int_{\sigma(T)}^{T}(\norm{\rho\dot{u}}_{L^2}^2+\norm{\nabla u}_{L^2}^2)(\norm{\rho\dot{u}}_{L^6}^2+\norm{\nabla u}_{L^2}^2+\norm{P-\bar{P}}_{L^6}^2)\\
     &\leq C(\tilde{\rho})\int_{\sigma(T)}^{T}(\norm{\sqrt{\rho}\dot{u}}_{L^2}^2+\norm{\nabla u}_{L^2}^2)(\norm{\nabla\dot{u}}_{L^2}^2+\norm{\nabla u}_{L^2}^4+\norm{\nabla u}_{L^2}^2+\norm{P-\bar{P}}_{L^6}^2)\\
     &\leq C(\tilde{\rho})(A_1(T)A_2(T)+A_2^2(T)+A_1^3(T)+A_1^2(T)+A_1(T)\mathcal{E}_0^{\frac{1}{3}})+C(\tilde{\rho})(A_1^2(T)+A_1(T)+\mathcal{E}_0^{\frac{1}{3}})E_0\\
     &\leq C(\tilde{\rho})\mathcal{E}_0^{\frac{1}{3}}(1+E_0)
     \end{aligned}
   \end{equation*}
   provided $\mathcal{E}_0\leq\min\{1,(C(\tilde{\rho}))^{-1}\}=\hat{\epsilon}_4$.
   
   Therefore, by choosing $N_0=C(\tilde{\rho})\mathcal{E}_0^{\frac{1}{3}}(1+E_0)$, $N_1=\frac{\tilde{\rho}P(\tilde{\rho})}{2\mu+\lambda}$ and $\zeta_0=\tilde{\rho}$ in Lemma \ref{lem-Zlotnik-inequality}, we have from Lemma \ref{lem-Zlotnik-inequality}, \eqref{rho-bound-small-T} and \eqref{b-estimate2} that
   \begin{equation}\label{rho-bound-large-T}
     \sup_{t\in[\sigma(T),T]}\norm{\rho}_{L^{\infty}}\leq\frac{3\tilde{\rho}}{2}
     +C(\tilde{\rho})\mathcal{E}_0^{\frac{1}{3}}(1+E_0)\leq\frac{7\tilde{\rho}}{4},
   \end{equation}
   provided $\mathcal{E}_0\leq\min\{\hat{\epsilon}_4,(\frac{\tilde{\rho}}{4C(\tilde{\rho})(1+E_0)})^3\}$. Then, combining \eqref{rho-bound-small-T} and \eqref{rho-bound-large-T}, we complete the proof of Lemma \ref{lem-rho-bound}.
\end{proof}

With Proposition \ref{prop-a-priori-estimate} well-prepared, we are now in a position to prove the following result concerning the exponential decay rate of classical solutions.

\begin{pro}\label{prop-energy-decay}
Assume $\bar{\rho}\leq 1$ and $\frac{\tilde{\rho}}{\bar{\rho}}\geq3$. Then for any $\gamma\in(1,\frac{3}{2}]$, $r\in[1,\infty)$ and $p\in[1,6]$, there exist two positive constants $C$ and $\eta_0$ with $C$ depending only on $\mu,\lambda,\gamma,a,\tilde{\rho},\bar{\rho},M,\Omega,r,p$ and the matrix $A$, and $\eta_0$ depending only on $\mu,\lambda,a,\Omega,\tilde{\rho},r,p$ and $\frac{\tilde{\rho}}{\bar{\rho}}$, but independent of $\gamma-1$, such that for any $t\geq1$, it holds that
\begin{equation}\label{energy-decay-prop}
  \norm{\rho-\bar{\rho}}_{L^r}+\norm{u}_{W^{1,p}}+\norm{\sqrt{\rho}\dot{u}}_{L^2}\leq Ce^{-\eta_0\bar{\rho}^{\gamma}t}.
\end{equation} 
\end{pro}

\begin{proof}
  First, by virtue of the mass conservation, we have
  \begin{equation*}
    \bar{\rho}=\frac{1}{|\Omega|}\int_{\Omega}\rho dx=\bar{\rho}_0.
  \end{equation*}
  Then, by the convexity of $P(\rho)$ and \eqref{essential-energy}, it holds that
  \begin{equation}\label{bound-rho-average}
    P(\bar{\rho})\leq\bar{P}(\rho)\leq |\Omega|^{-1}(\gamma-1)E_0\leq |\Omega|^{-1}\mathcal{E}_0,\quad\textrm{i.e.},\quad\bar{\rho}\leq (a|\Omega|)^{-\frac{1}{\gamma}}\mathcal{E}_0^{\frac{1}{\gamma}}.
  \end{equation}
  Here we shall give the relationship among $(\rho-\bar{\rho})^2$, $G(\rho)$ and $(P(\rho)-P(\bar{\rho}))(\rho-\bar{\rho})$ with 
  \begin{equation*}
    G(\rho)=\rho\int_{\bar{\rho}}^{\rho}\frac{P(s)-P(\bar{\rho})}{s^2}ds
  \end{equation*}
  for any $\rho\in[0,2\tilde{\rho}]$ and $\gamma\in(1,\frac{3}{2}]$. Note that $\bar{\rho}\leq1$ and $\bar{\rho}\ll\tilde{\rho}$ if $\mathcal{E}_0\leq a|\Omega|$ is small enough.
  
  A direct analysis on three terms above yields that (see Lemma \ref{lem-relationship} in Appendix for details)
  \begin{equation}\label{relationship-inequality}
  \begin{aligned}
  &\frac{P(\bar{\rho})}{\bar{\rho}}(\rho-\bar{\rho})^2\leq(P(\rho)-P(\bar{\rho}))(\rho-\bar{\rho}),\\
  &(\rho-\bar{\rho})^2\leq\frac{1}{C_1}\tilde{\rho}\bar{\rho}^{1-\gamma}G(\rho),\\
  &\bar{\rho}G(\rho)\leq(P(\rho)-P(\bar{\rho}))(\rho-\bar{\rho}).
  \end{aligned}
\end{equation}
with constant $C_1>0$ depending only on $a$ and $\frac{\tilde{\rho}}{\bar{\rho}}$.

Similar to \eqref{essential-energy-estimate}, we have
\begin{equation}\label{essential-energy-modify1}
  \frac{d}{dt}\int(\frac{1}{2}\rho|u|^2+G(\rho))+\phi(t)\leq0
\end{equation}  
with 
\begin{equation*}
  \phi(t)=(2\mu+\lambda)\norm{\div u}_{L^2}^2+\mu\norm{\curl u}_{L^2}^2.
\end{equation*}
Then multiplying $\eqref{Large-CNS-eq}_2$ by $\mathcal{B}[\rho-\bar{\rho}]$, we have fromLemma \ref{lem-Bogovskii-operator} and Proposition \ref{prop-a-priori-estimate} that
\begin{equation}\label{rho-L^2-bound1}
  \begin{aligned}
  &\quad\int(P(\rho)-P(\bar{\rho}))(\rho-\bar{\rho})\\
  &=\frac{d}{dt}\int\rho u\cdot\mathcal{B}[\rho-\bar{\rho}]-\int\rho u\cdot\nabla\mathcal{B}[\rho-\bar{\rho}]\cdot u+\int\rho u\cdot\mathcal{B}[\div(\rho u)]+\mu\int\nabla u\cdot\nabla\mathcal{B}[\rho-\bar{\rho}]\\
  &\quad+(\mu+\lambda)\int(\rho-\bar{\rho})\div u\\
  &\leq\frac{d}{dt}\int\rho u\cdot\mathcal{B}[\rho-\bar{\rho}]+C\norm{\sqrt{\rho}u}_{L^4}^2\norm{\rho-\bar{\rho}}_{L^2}
  +C\norm{\rho u}_{L^2}^2+C\norm{\nabla u}_{L^2}\norm{\rho-\bar{\rho}}_{L^2}\\
  &\leq\frac{d}{dt}\int\rho u\cdot\mathcal{B}[\rho-\bar{\rho}]+C(\tilde{\rho})\norm{\rho^{\frac{1}{3}}u}_{L^3}\norm{u}_{L^6}\norm{\rho-\bar{\rho}}_{L^2}+C(\tilde{\rho})\norm{\nabla u}_{L^2}^2+C\norm{\nabla u}_{L^2}\norm{\rho-\bar{\rho}}_{L^2}\\
  &\leq\frac{d}{dt}\int\rho u\cdot\mathcal{B}[\rho-\bar{\rho}]+\frac{P(\bar{\rho})}{2\bar{\rho}}\norm{\rho-\bar{\rho}}_{L^2}^2+C(\tilde{\rho})\frac{\bar{\rho}}{P(\bar{\rho})}\norm{\nabla u}_{L^2}^2,
  \end{aligned}
\end{equation}
which together with \eqref{relationship-inequality} and \eqref{curl-effective-viscous-estimate1} yields
\begin{equation}\label{rho-L^2-bound2}
  \begin{aligned}
  \int(P(\rho)-P(\bar{\rho}))(\rho-\bar{\rho})&\leq\frac{d}{dt}\int2\rho u\cdot\mathcal{B}[\rho-\bar{\rho}]+C(\tilde{\rho})\bar{\rho}^{1-\gamma}\norm{\nabla u}_{L^2}^2\\
  &\leq\frac{d}{dt}\int2\rho u\cdot\mathcal{B}[\rho-\bar{\rho}]+C_1(\tilde{\rho})\bar{\rho}^{1-\gamma}\phi(t).
  \end{aligned}
\end{equation}
Moreover, it follows from \eqref{relationship-inequality} that there exists a constant $C_2(\tilde{\rho})>0$ such that
\begin{equation}\label{W-transition1}
  |\int2\rho u\cdot\mathcal{B}[\rho-\bar{\rho}]|\leq C_2(\tilde{\rho})\bar{\rho}^{\frac{1-\gamma}{2}}\int(\frac{1}{2}\rho|u|^2+G(\rho)), 
\end{equation}
which combined with \eqref{essential-energy-modify1} and \eqref{rho-L^2-bound2} gives that
\begin{equation}\label{essential-energy-modify2}
  \frac{d}{dt}W(t)+\delta_1\bar{\rho}^{\gamma-1}\int(P(\rho)-P(\bar{\rho}))(\rho-\bar{\rho})+\frac{1}{2}\phi(t)\leq0
\end{equation}
with
\begin{equation*}
  W(t)=\int(\frac{1}{2}\rho|u|^2+G(\rho))-\delta_1\bar{\rho}^{\gamma-1}\int2\rho u\cdot\mathcal{B}[\rho-\bar{\rho}],\quad\delta_1=\min\left\{\frac{1}{2C_1(\tilde{\rho})},\frac{1}{2C_2(\tilde{\rho})}\right\}
\end{equation*}
satisfying
\begin{equation}\label{W-bound}
  \frac{1}{2}\int(\frac{1}{2}\rho|u|^2+G(\rho))\leq W(t)\leq\frac{3}{2}\int(\frac{1}{2}\rho|u|^2+G(\rho)).
\end{equation}
due to \eqref{W-transition1} and $\bar{\rho}\leq1$.

By virtue of Poincar\'{e} inequality \eqref{Poincare-inequality} and \eqref{curl-effective-viscous-estimate1}, we have
\begin{equation*}
  \int\rho|u|^2\leq C(\tilde{\rho})\norm{\nabla u}_{L^2}\leq C_3(\tilde{\rho})\phi(t),
\end{equation*}
which together with \eqref{relationship-inequality} implies that
\begin{equation}\label{diffusion-compare1}
  \delta_1\bar{\rho}^{\gamma-1}\int(P(\rho)-P(\bar{\rho}))(\rho-\bar{\rho})+\frac{1}{2}\phi(t)\geq\delta_2\bar{\rho}^{\gamma}W(t)
\end{equation}
with $\delta_2=\frac{2}{3}\min\{\delta_1,\frac{1}{C_3(\tilde{\rho})}\}$.

Thus we deduce from \eqref{essential-energy-modify2} and \eqref{diffusion-compare1} that
\begin{equation*}
  \frac{d}{dt}W(t)+\delta_2\bar{\rho}^{\gamma}W(t)\leq 0,
\end{equation*} 
which combined with \eqref{W-bound} yields that
\begin{equation}\label{energy-decay1}
  \int(\frac{1}{2}\rho|u|^2+G(\rho))\leq 2E_0e^{-\delta_2\bar{\rho}^{\gamma}t}.
\end{equation}
Furthermore, it holds from \eqref{essential-energy-modify1} that for $0<\delta_3<\delta_2$,
\begin{equation}\label{energy-decay1-byproduct}
  \int_0^{\infty}\phi(t)e^{\delta_3\bar{\rho}^{\gamma}t}dt\leq CE_0.
\end{equation}
Choosing $m=0$ in \eqref{A-control2}, and using \eqref{small-time-estimate3}, \eqref{curl-effective-viscous-estimate1} and \eqref{curl-effective-viscous-estimate2}, we have
\begin{equation}\label{1rd-energy-modify1}
  \begin{aligned}
  &\quad\frac{d}{dt}\left(\phi(t)+\mu\int_{\pa\Omega}Au\cdot u-\int2(P(\rho)-P(\bar{\rho}))\div u\right)+\int\rho|\dot{u}|^2\\
  &\leq C(\tilde{\rho},M)\norm{\nabla u}_{L^2}^2+C\norm{\nabla u}_{L^2}^{\frac{3}{2}}(\norm{\rho\dot{u}}_{L^2}+\norm{\nabla u}_{L^2}+\norm{P(\rho)-P(\bar{\rho})}_{L^6})^{\frac{3}{2}}\\
  &\leq\frac{1}{2}\norm{\sqrt{\rho}\dot{u}}_{L^2}^2+C(\tilde{\rho},M)\norm{\nabla u}_{L^2}^2+C(\tilde{\rho})\norm{P(\rho)-P(\bar{\rho})}_{L^2}^2\\
  &\leq\frac{1}{2}\norm{\sqrt{\rho}\dot{u}}_{L^2}^2+C(\tilde{\rho},M)\phi(t)+C(\tilde{\rho})\int G(\rho), 
  \end{aligned}
\end{equation}
where we have used from \eqref{relationship-inequality}, \eqref{k-another-bound2} and \eqref{k-another-bound3} that\footnote{Indeed, from \eqref{relationship-inequality}, it holds that $\norm{\rho-\tilde{\rho}}_{L^2}^2\leq C(\tilde{\rho})\bar{\rho}^{1-\gamma}\int G(\rho)$. By \eqref{small-assumption3} and \eqref{bound-rho-average}, it is easy to check that $(\gamma-1)^{\frac{1}{17}}E_0\leq C(\tilde{\rho})$ and hence $\bar{\rho}\leq C(\tilde{\rho})(\gamma-1)^{\frac{16}{17\gamma}}$. If $\bar{\rho}=C(\gamma-1)^{\eta}$, then $\bar{\rho}^{1-\gamma}$ is bounded as $\gamma\rightarrow1$ since $(\gamma-1)^{\gamma-1}\rightarrow1$ when $\gamma-1\rightarrow0$.}
\begin{equation}\label{P-L^2-estimate}
  \norm{P(\rho)-P(\bar{\rho})}_{L^2}^2\leq C(\tilde{\rho})\int(P(\rho)-P(\bar{\rho}))(\rho-\bar{\rho})\leq C(\tilde{\rho})\int G(\rho).
\end{equation}
By multiplying \eqref{1rd-energy-modify1} by $e^{\delta_3\bar{\rho}^{\gamma}t}$, we have
\begin{equation}\label{1rd-energy-modify2}
  \begin{aligned}
  &\quad\frac{d}{dt}\left(e^{\delta_3\bar{\rho}^{\gamma}t}\phi(t)+\mu e^{\delta_3\bar{\rho}^{\gamma}t}\int_{\pa\Omega}Au\cdot u-e^{\delta_3\bar{\rho}^{\gamma}t}\int2(P(\rho)-P(\bar{\rho}))\div u\right)+\frac{1}{2}e^{\delta_3\bar{\rho}^{\gamma}t}\int\rho|\dot{u}|^2\\
  &\leq C(\tilde{\rho},M)e^{\delta_3\bar{\rho}^{\gamma}t}(\phi(t)+\int G(\rho)),
  \end{aligned}
\end{equation}
where we have used the facts that 
\begin{equation*}
  0\leq\int_{\pa\Omega}Au\cdot u\leq C\norm{\nabla u}_{L^2}^2\leq C\phi(t),
\end{equation*}
and 
\begin{equation*}
  \left|\int2(P(\rho)-P(\bar{\rho}))\div u\right|\leq C(\tilde{\rho})\int G(\rho)+\frac{1}{2}\phi(t),
\end{equation*} due to \eqref{P-L^2-estimate}.
This together with \eqref{energy-decay1} and \eqref{energy-decay1-byproduct} yields that
\begin{equation}\label{energy-decay2}
  \norm{\nabla u}_{L^2}^2\leq C(\tilde{\rho},M)e^{-\delta_3\bar{\rho}^{\gamma}t}
\end{equation}
and
\begin{equation}\label{energy-decay2-byproduct}
  \int_0^{\infty}e^{\delta_3\bar{\rho}^{\gamma}t}\norm{\sqrt{\rho}\dot{u}}_{L^2}^2\leq C(\tilde{\rho},M).
\end{equation}
A similar analysis based on \eqref{A-control-3} for $m=3$, \eqref{boundary-term-A2-transition2}, Lemma \ref{lem-small-time-estimate}, \eqref{P-L^2-estimate}, \eqref{energy-decay2} and \eqref{energy-decay2-byproduct} gives that
\begin{equation}\label{energy-decay3}
  \norm{\sqrt{\rho}\dot{u}}_{L^2}^2\leq Ce^{-\delta_3\bar{\rho}^{\gamma}t},
\end{equation}
which together with \eqref{energy-decay1}, \eqref{energy-decay2}, \eqref{relationship-inequality} and \eqref{curl-effective-viscous-estimate6} yields \eqref{energy-decay-prop} and finishes the proof of Proposition \ref{prop-energy-decay}.  
  
\end{proof}

\noindent\textbf{Proof of Theorem \ref{thm-global-CNS}} In the following, we will prove the main results of this paper. First of all, we derive the time-dependent higher-order estimates of the smooth solution $(\rho, u)$.
From now on, we will always assume that \eqref{small-assumption4} holds and denote the positive constant by $C$ depending on
\begin{equation*}
  T,\,\norm{g}_{L^2},\,\norm{\nabla u_0}_{H^1},\,\norm{\rho_0}_{W^{2,q}},\,\norm{P(\rho)}_{W^{2,q}},
\end{equation*}
for $q\in(3,6)$, as well as $\mu,\lambda,\gamma,a,\tilde{\rho},\Omega,M$ and the matrix $A$, where $g$ is given in \eqref{compatibility-condition}. Here we only sketch the higher-order estimates in the following lemma, which have been proved in \cite{Cai-Li2023}.

\begin{lem}\label{lem-higher-order-estimate}
  Under the conditions of Theorem \ref{thm-global-CNS}, it holds that
  \begin{equation*}
    \begin{aligned}
    &\sup_{t\in[0,T]}\int\rho|\dot{u}|^2+\int_0^T\norm{\nabla\dot{u}}_{L^2}^2\leq C,\\
    &\sup_{t\in[0,T]}(\norm{\nabla\rho}_{L^6}+\norm{u}_{H^2})+\int_0^T(\norm{\nabla u}_{L^{\infty}}+\norm{\nabla^2u}_{L^6}^2)\leq C,\\
    &\sup_{t\in[0,T]}\norm{\sqrt{\rho}u_t}_{L^2}^2+\int_0^T\norm{\nabla u_t}_{L^2}^2\leq C,\\
    &\sup_{t\in[0,T]}(\norm{\rho}_{H^2}+\norm{P}_{H^2})\leq C,\\
    &\sup_{t\in[0,T]}(\norm{\rho_t}_{H^1}+\norm{P_t}_{H^1})+\int_0^T(\norm{\rho_{tt}}_{L^2}^2+\norm{P_{tt}}_{L^2}^2)\leq C,\\
    &\sup_{t\in[0,T]}\sigma\norm{\nabla u_t}_{L^2}^2+\int_0^T\sigma\norm{\sqrt{\rho}u_{tt}}_{L^2}\leq C,\\
    &\sup_{t\in[0,T]}\sigma\norm{\nabla u}_{H^2}^2+\int_0^T(\norm{\nabla u}_{H^2}^2+\norm{\nabla^2u}_{W^{1,q}}^{p_0}+\sigma\norm{\nabla u_t}_{H^1}^2)\leq C,\\
    &\sup_{t\in[0,T]}(\norm{\rho}_{W^{2,q}}+\norm{P}_{W^{2,q}})\leq C,\\
    &\sup_{t\in[0,T]}\sigma(\norm{\nabla u_t}_{H^1}+\norm{\nabla u}_{W^{2,q}})+\int_0^T\sigma^2\norm{\nabla u_{tt}}_{L^2}^2\leq C,
    \end{aligned}
  \end{equation*}
  for $q\in(3,6)$ and $p_0=\frac{9q-6}{10q-12}\in(1,\frac{7}{6})$.
\end{lem}
Thus combining Proposition \ref{prop-a-priori-estimate}, Proposition \ref{prop-energy-decay} with the higher-order estimates above as well as the local existence in Lemma \ref{lem-local-sol}, we can prove Theorem \ref{thm-global-CNS} by similar arguments as in \cite{Cai-Li2023}. Here, we omit the details for brevity.

\appendix
\section{The mathematical analysis on three terms about density.}
In this appendix, we will give a mathematical analysis on the precise relationship among $(\rho-\bar{\rho})^2$, $G(\rho)$ and $(P(\rho)-P(\bar{\rho}))(\rho-\bar{\rho})$.

\begin{lem}\label{lem-relationship}
  There exists a clear relationship among $(\rho-\bar{\rho})^2$, $G(\rho)$ and $(P(\rho)-P(\bar{\rho}))(\rho-\bar{\rho})$ for any $\rho\in[0,2\tilde{\rho}]$ and $\gamma\in(1,\frac{3}{2}]$. If $\bar{\rho}\ll\tilde{\rho}$, then we obtain
  \begin{equation}\label{relationship}
  \begin{aligned}
  &\frac{P(\bar{\rho})}{\bar{\rho}}(\rho-\bar{\rho})^2\leq(P(\rho)-P(\bar{\rho}))(\rho-\bar{\rho}),\\
  &(\rho-\bar{\rho})^2\leq\frac{1}{C_1}\tilde{\rho}\bar{\rho}^{1-\gamma}G(\rho),\\
  &\bar{\rho}G(\rho)\leq(P(\rho)-P(\bar{\rho}))(\rho-\bar{\rho}).
  \end{aligned}
\end{equation}
with constant $C_1>0$ depending only on $a$ and $\frac{\tilde{\rho}}{\bar{\rho}}$. In fact, it suffices to assume $\frac{\tilde{\rho}}{\bar{\rho}}\geq 3$ here.
\end{lem}

\begin{proof}
 Here we assume $\bar{\rho}\ll\tilde{\rho}$, which means $\bar{\rho}$ is much smaller than $\tilde{\rho}$. First set
\begin{equation*}
  f(\rho)=\frac{(P(\rho)-P(\bar{\rho}))(\rho-\bar{\rho})}{(\rho-\bar{\rho})^2}=\frac{P(\rho)-P(\bar{\rho})}{\rho-\bar{\rho}},
\end{equation*}
and let
\begin{equation*}
  f(\bar{\rho})=\lim_{\rho\rightarrow\bar{\rho}}f(\rho)=P'(\bar{\rho}).
\end{equation*}
Then a direct calculation yields that
\begin{equation}\label{f-1rd-derivative}
  f'(\rho)=\frac{P'(\rho)(\rho-\bar{\rho})-(P(\rho)-P(\bar{\rho}))}{(\rho-\bar{\rho})^2}
  =(\rho-\bar{\rho})^{-2}\int_{\bar{\rho}}^{\rho}(P'(\rho)-P'(s))ds\geq0,
\end{equation}
which implies
\begin{equation}\label{f-bound}
  f(\rho)\in[f(0),f(\tilde{\rho})]=[\frac{P(\bar{\rho})}{\bar{\rho}},f(\tilde{\rho})].
\end{equation}

We then define
\begin{equation*}
  h(\rho)=\frac{G(\rho)}{(\rho-\bar{\rho})^2}
\end{equation*}
and also set
\begin{equation*}
  h(\bar{\rho})=\lim_{\rho\rightarrow\bar{\rho}}h(\rho)=\frac{1}{2}G''(\bar{\rho})=\frac{P'(\bar{\rho})}{2\bar{\rho}}.
\end{equation*}
A similar calculation gives that
\begin{equation}\label{h-1rd-derivative}
  h'(\rho)=\frac{G'(\rho)(\rho-\bar{\rho})-2G(\rho)}{(\rho-\bar{\rho})^3}=\frac{h_1(\rho)}{(\rho-\bar{\rho})^3},
\end{equation}
It is easy to verify that
\begin{equation*}
  h_1(\bar{\rho})=0
\end{equation*}
and
\begin{equation}\label{h1-1rd-derivative}
  h_1'(\rho)=G''(\rho)(\rho-\bar{\rho})-G'(\rho)=\int_{\bar{\rho}}^{\rho}(G''(\rho)-G''(s))ds\leq0
\end{equation}
due to the decreasing monotonicity of $G''(\rho)=\frac{P'(\rho)}{\rho}$ for $\gamma\in(1,\frac{3}{2}]$.

Then we obtain
\begin{equation*}
  h_1(\rho)=\begin{cases}
              >0, & \mbox{if }\rho<\bar{\rho}, \\
              <0, & \mbox{if }\rho>\bar{\rho},
            \end{cases}
\end{equation*}
which implies
\begin{equation*}
  h'(\rho)=\begin{cases}
              <0, & \mbox{if }\rho<\bar{\rho}, \\
              =\frac{1}{6}G'''(\bar{\rho})<0, &\mbox{if }\rho=\bar{\rho},\\
              <0, & \mbox{if }\rho>\bar{\rho}.
            \end{cases}
\end{equation*}
That means 
\begin{equation}\label{h-bound1}
  h(\rho)\in[h(\tilde{\rho}),h(0)]=[h(\tilde{\rho}),\frac{P(\bar{\rho})}{\bar{\rho}^2}].
\end{equation}
Here, we need to estimate $h(\tilde{\rho})$. Since $\bar{\rho}\ll\tilde{\rho}$, it holds that\footnote{The focus of this estimate is at the case $\gamma\rightarrow1$, so we need keep some intrinsical relation unchanged during the calculation as $\gamma$ tending to 1.}
\begin{equation}\label{h-bound-transition1}
  \begin{aligned}
  h(\tilde{\rho})&=\left(1-\frac{\bar{\rho}}{\tilde{\rho}}\right)^{-2}\tilde{\rho}^{-1}a\left(\frac{1}{\gamma-1}(\tilde{\rho}^{\gamma-1}-\bar{\rho}^{\gamma-1})
  +\bar{\rho}^{\gamma}(\tilde{\rho}^{-1}-\bar{\rho}^{-1})\right)\\
  &=a(1-A^{-1})^{-2}\left(\frac{1}{\gamma-1}A^{-(2-\gamma)}-\frac{\gamma}{\gamma-1}A^{-1}+A^{-2}\right)\bar{\rho}^{\gamma-2}\\
  &=a(1-A^{-1})^{-2}\left(\frac{1}{\gamma-1}(A^{\gamma-1}-\gamma)A^{-1}+A^{-2}\right)\bar{\rho}^{\gamma-2}\\
  &\geq a(1-A^{-1})^{-2}\left((\ln A-1)A^{-1}+A^{-2}\right)\bar{\rho}^{\gamma-2}\\
  &\geq\frac{a(\ln A-1)}{(1-A^{-1})^2}A^{-1}\bar{\rho}^{\gamma-2}=\frac{a(\ln A-1)}{(1-A^{-1})^2}\frac{\bar{\rho}^{\gamma-1}}{\tilde{\rho}}
  \end{aligned}
\end{equation}
with $A=\frac{\tilde{\rho}}{\bar{\rho}}\gg1$ (actually it suffices to set $A\geq 3$). Then there exists a constant $C_1>0$ depending only on $a$ and $\frac{\tilde{\rho}}{\bar{\rho}}$, such that
\begin{equation}\label{h-bound2}
  h(\rho)\in[C_1\frac{\bar{\rho}^{\gamma-1}}{\tilde{\rho}},\frac{P(\bar{\rho})}{\bar{\rho}^2}].
\end{equation}

Finally, define
\begin{equation*}
  k(\rho)=\frac{(P(\rho)-P(\bar{\rho}))(\rho-\bar{\rho})}{G(\rho)}=\frac{F(\rho)}{G(\rho)}
\end{equation*}
and also set
\begin{equation*}
  k(\bar{\rho})=\lim_{\rho\rightarrow\bar{\rho}}k(\rho)=\frac{2P'(\bar{\rho})}{P'(\bar{\rho})\bar{\rho}^{-1}}=2\bar{\rho}.
\end{equation*}
Then an analogous computation gives that
\begin{equation}\label{k-1rd-derivative}
  k'(\rho)=\frac{F'(\rho)G(\rho)-F(\rho)G'(\rho)}{G^2(\rho)}=\frac{k_1(\rho)}{G^2(\rho)}.
\end{equation}
It is easy to verify that
\begin{equation*}
  k_1(\bar{\rho})=0
\end{equation*}
and
\begin{equation}\label{k1-1rd-derivative}
  \begin{aligned}
  k_1'(\rho)&=F''(\rho)G(\rho)-F(\rho)G''(\rho)\\
  &=[P''(\rho)(\rho-\bar{\rho})+2P'(\rho)][\frac{1}{\gamma-1}P(\rho)+P(\bar{\rho})(1-\frac{\gamma}{\gamma-1}\frac{\rho}{\bar{\rho}})]\\
  &\quad-(P(\rho)-P(\bar{\rho}))(\rho-\bar{\rho})\frac{P'(\rho)}{\rho}\\
  &=\frac{\gamma}{\gamma-1}P(\bar{\rho})P''(\rho)(\rho-\bar{\rho})(1-\frac{\rho}{\bar{\rho}})+2P'(\rho)G(\rho)\\
  &=\frac{2P'(\rho)}{\rho}[\rho G(\rho)-\frac{\gamma}{2}\frac{P(\bar{\rho})}{\bar{\rho}}(\rho-\bar{\rho})^2]=\frac{2P'(\rho)}{\rho}k_2(\rho).
  \end{aligned}
\end{equation}
Obviously, $k_2(\rho)$ satisfies that
\begin{equation*}
  k_2(\bar{\rho})=k_2'(\bar{\rho})=k_2''(\bar{\rho})=0,
\end{equation*}
and
\begin{equation}\label{k2-derivative}
  \begin{aligned}
  k_2'(\rho)&=\rho G'(\rho)+G(\rho)-P'(\bar{\rho})(\rho-\bar{\rho}),\\
  k_2''(\rho)&=\rho G''(\rho)+2G'(\rho)-P'(\bar{\rho})=2G'(\rho)+P'(\rho)-P'(\bar{\rho}),\\
  k_2'''(\rho)&=2G''(\rho)+P''(\rho)=2\frac{P'(\rho)}{\rho}+P''(\rho)\geq0.
  \end{aligned}
\end{equation}
That means
\begin{equation*}
  \begin{aligned}
  k_2''(\rho)&=\begin{cases}
                <0, & \mbox{if }\rho<\bar{\rho}, \\
                >0, & \mbox{if }\rho>\bar{\rho},
              \end{cases}\\
  k_2'(\rho)&\geq0,\\
  k_2(\rho)&=\begin{cases}
                <0, & \mbox{if }\rho<\bar{\rho}, \\
                >0, & \mbox{if }\rho>\bar{\rho},
              \end{cases}            
  \end{aligned}
\end{equation*}
which implies
\begin{equation*}
  k_1(\rho)\geq 0,\quad k'(\rho)\geq0.
\end{equation*}
Thus we get the bound of $k(\rho)$ as
\begin{equation}\label{k-bound}
  k(\rho)\in[k(0),k(\tilde{\rho})]=[\bar{\rho},k(\tilde{\rho})].
\end{equation}
Therefore, we conclude from \eqref{f-bound}, \eqref{h-bound2} and \eqref{k-bound} that
\begin{equation*}
  \begin{aligned}
  &\frac{P(\bar{\rho})}{\bar{\rho}}(\rho-\bar{\rho})^2\leq(P(\rho)-P(\bar{\rho}))(\rho-\bar{\rho}),\\
  &(\rho-\bar{\rho})^2\leq\frac{1}{C_1}\tilde{\rho}\bar{\rho}^{1-\gamma}G(\rho),\\
  &\bar{\rho}G(\rho)\leq(P(\rho)-P(\bar{\rho}))(\rho-\bar{\rho}).
  \end{aligned}
\end{equation*}
and complete the proof of Lemma \ref{lem-relationship}.
\end{proof}

\begin{rem}\label{rem-relationship}
  Indeed, we can still show the following relationship between $(P(\rho)-P(\bar{\rho}))(\rho-\bar{\rho})$ and $G(\rho)$. As in \eqref{k-bound}, for any $\rho\in[0,\tilde{\rho}]$,
  \begin{equation}\label{k-another-bound1}
    (P(\rho)-P(\bar{\rho}))(\rho-\bar{\rho})\leq k(\tilde{\rho})G(\rho).
  \end{equation}
  Here we aim to determine the upper bound of $k(\tilde{\rho})$. Just as done in \eqref{h-bound-transition1}, we have
  \begin{equation}\label{k-another-bound-transition1}
    \begin{aligned}
    k(\tilde{\rho})&=\frac{\bar{\rho}^{\gamma+1}A^{\gamma+1}(1-A^{-\gamma})(1-A^{-1})}{\bar{\rho}^{\gamma}
    A\left(\frac{1}{\gamma-1}(A^{\gamma-1}-\gamma)+A^{-1}\right)}\\
    &=\frac{\bar{\rho}A^{\gamma}(1-A^{-\gamma})(1-A^{-1})}{\frac{1}{\gamma-1}(A^{\gamma-1}-\gamma)+A^{-1}}\\
    &\leq\frac{\bar{\rho}A^{\gamma}(1-A^{-\gamma})(1-A^{-1})}{\ln A-1+A^{-1}}\\
    &=\frac{\tilde{\rho}A^{\gamma-1}(1-A^{-\gamma})(1-A^{-1})}{\ln A-1+A^{-1}}
    \end{aligned}
  \end{equation}
  with $A=\frac{\tilde{\rho}}{\bar{\rho}}\geq3$ and $\gamma\in(1,\frac{3}{2}]$. However, if choosing $A\gg1$ as $\gamma\rightarrow1$, we can get a better estimate as
  \begin{equation}\label{k-another-bound-transition2}
    \begin{aligned}
    k(\tilde{\rho})&=\frac{\bar{\rho}A^{\gamma}(1-A^{-\gamma})(1-A^{-1})}{\frac{1}{\gamma-1}(A^{\gamma-1}-\gamma)+A^{-1}}\\
    &\leq\frac{\bar{\rho}A^{\gamma}(1-A^{-\gamma})(1-A^{-1})}{\frac{1}{3(\gamma-1)}A^{\gamma-1}}\\
    &\leq3(\gamma-1)\tilde{\rho}.
    \end{aligned}
  \end{equation}
  under the condition
  \begin{equation}\label{A-restriction}
    A\geq3^{\frac{1}{\gamma-1}}
  \end{equation}
  which implies
  \begin{equation*}
    A^{\gamma-1}\geq3\geq2\gamma
  \end{equation*}
  for any $\gamma\in(1,\frac{3}{2}]$. 
  
  Thus we conclude that if $A=\frac{\tilde{\rho}}{\bar{\rho}}\geq3^{\frac{1}{\gamma-1}}$, then for any $\gamma\in(1,\frac{3}{2}]$,
  \begin{equation}\label{k-another-bound2}
    (P(\rho)-P(\bar{\rho}))(\rho-\bar{\rho})\leq 3\tilde{\rho}(\gamma-1)G(\rho),
  \end{equation}
  and if $A=\frac{\tilde{\rho}}{\bar{\rho}}\in[3,3^{\frac{1}{\gamma-1}}]$ with $\gamma\in(1,\frac{3}{2}]$, it holds from \eqref{k-another-bound-transition1} that
  \begin{equation}\label{k-another-bound3}
    (P(\rho)-P(\bar{\rho}))(\rho-\bar{\rho})\leq \frac{3\tilde{\rho}}{\ln3-1}G(\rho).
  \end{equation}
\end{rem}

\medskip

\section*{\bf Data availability}
No data was used for the research described in the article.

\section*{\bf Conflicts of interest}

The authors declare no conflict of interest.

\end{document}